\documentclass[12pt]{article}
\usepackage[T1]{fontenc}
\usepackage[latin1]{inputenc}
\usepackage[dvips]{graphicx}
\usepackage[francais]{babel}
\usepackage{amsmath}
\usepackage{amsthm}
\usepackage{amssymb}
\usepackage{amsfonts}
\usepackage{mathrsfs,bbm}
\usepackage{mathrsfs}
\usepackage[all]{xy}
\usepackage{color}
\usepackage{pdfcolmk}
\usepackage{comment}
\usepackage{url}

\numberwithin{equation}{section}

\theoremstyle{plain}
\newtheorem{theoreme}{Theor\`eme}

\newtheorem{corollaire}{Corollaire}

\newtheorem{lemme}{Lemme}
\newtheorem{proposition}{Proposition}

\theoremstyle{definition}
\newtheorem{definition}{Définition}
\newtheorem{exemple}{Exemple}
\newtheorem*{notation}{Notation}
\newtheorem{remarque}{Remarque}

\newtheorem*{acknowledgment}{Remerciements}

\theoremstyle{remark}

\newcommand{\R}{\mathbb{R}}
\newcommand{\C}{\mathbb{C}}

\newcommand{\Z}{\mathbb{Z}}
\newcommand{\Q}{\mathbb{Q}}
\newcommand{\N}{\mathbb{N}}
\newcommand{\F}{\mathbb{F}}
\newcommand{\A}{\mathbb{A}}

\newcommand{\pgcd}{\mathrm{pgcd}}
\newcommand{\id}{\mathrm{id}}

\newcommand{\SL}{\mathrm{SL}}
\newcommand{\GL}{\mathrm{GL}}

\newcommand{\Hom}{\mathrm{Hom}}
\newcommand{\End}{\mathrm{End}}
\newcommand{\Ens}{\mathfrak{Ens}}
\newcommand{\Spec}{\mathrm{Spec\ }}

\newcommand{\abs}[1]{\lvert#1\rvert}

\def \Ã©{\'e}
\def \Ã¨{\`{e}}
\def \Ã¹{\`{u}}
\def \Ã®{\^{\i}}
\def \Ã {\`{a}}
\def \Ã§{\c{c}}

\title{Compactification projective de $\Spec\Z$ \\ (d'après Durov)}
\author{Javier Fresán\footnote{Il s'agit d'une version très préliminaire d'un survey sur la thèse de Durov. Tout commentaire est bienvenu au courriel de l'auteur: javierfresan@gmail.com.}}

\begin{document}

\maketitle

\tableofcontents
%\newpage

\section{Introduction}

Parmi les nombreux outils mathematiques developpés par Gauss, certains ont eu un destin tout particulier. C'est le cas des $q$-entiers, introduits par lui afin d'évaluer ce que l'on appelle aujourd'hui les sommes quadratiques gaussiennes. Ils servent, par exemple, à formuler des analogues non commutatifs du théorème du binôme et ils jouent un rôle proéminent dans la mécanique quantique. Mais c'est leur apparition dans le cadre des géométries finies qui a ouvert une voie de recherche tout à fait inattendue. En effet, lorsque $q$ est une puissance du nombre premier $p,$ le cardinal des espaces projectifs sur le corps à $q$ éléments n'est autre que
\begin{equation*}
\abs{\mathbb{P}^{n-1}(\F_q)}=\frac{\abs{\mathbb{A}^n(\F_q)-\{0\}}}{\abs{\mathbb{G}_m(\F_q)}}=\frac{q^n-1}{q-1}=[n]_q,
\end{equation*} et l'on obtient aussi une formule pour le cardinal de la Grassmanienne: 
\begin{equation*}
\abs{\mathrm{Gr}(n,j)(\F_q)}=\abs{\{\mathbb{P}^j(\F_q)\subset \mathbb{P}^n(\F_q)\}}={n\choose j}_q=\frac{[n]_q!}{[j]_q![n-j]q!}. 
\end{equation*} La remarque que les expressions ci-dessus ont encore un sens lorsque $q=1$ a conduit Tits en 1957 à postuler l'existence d'un objet algébrique, \textit{le corps de caractéristique un}, sur lequel on pourrait écrire formellement
\begin{equation*}
\abs{\mathbb{P}^{n-1}(\F_1)}=[n]_1=n, \quad \abs{\mathrm{Gr}(n,j)(\F_1)}={n\choose j}_1={n\choose j}.
\end{equation*} Ainsi, $\mathbb{P}^{n-1}(\F_1)$ serait juste un ensemble fini $P$ à $n$ éléments, et $\mathrm{Gr}(n,j)(\F_1),$ l'ensemble des parties de $P$ ayant cardinal $j.$ En niant l'axiome d'après lequel une droite projective possède au moins trois points, on obtient des espaces où la géométrie devient combinatoire. Notamment, on peut imaginer $\GL_n(\F_1)$ comme le group symétrique $\mathcal{S}_n$ et $\SL_n(\F_1)$ comme étant formé par les permutations paires. Le programme de Tits consiste alors à interpréter les groupes de Weyl comme des groupes de Chevalley sur le corps à un élément \cite{Ti}.        

\smallskip
Il a fallu attendre jusqu'en 1994 pour la suite des événements. Dans ses conférences sur la fonction zêta et les motifs \cite{Ma1}, largement inspirées des travaux de Deninger et Kurokawa, Manin se réfère à $\Spec\F_1$ comme \textit{le point absolu} dont on aurait besoin pour reproduire la preuve de Weil de l'hypothèse de Riemann pour les courbes sur un corps fini. Les fonctions zêta des variétés sur $\F_1$ auraient la forme la plus simple, comme celle-ci: $\zeta_{\mathbb{P}^N(\F_1)}(s)=s(s-1)\ldots(s-N).$ En essayant de donner un sens au produit infini régularisé
\begin{equation*}
\frac{(2\pi)^s}{\Gamma(s)}=\prod_{n\geq 0} \frac{s+n}{2\pi},
\end{equation*} Manin a conjecturé l'existence d'une catégorie de motifs sur $\F_1$ et encouragé les mathématiciens à développer une géométrie algébrique sur le corps à un élément.         

\smallskip
 
Même si cette suggestion a pris forme dans de très diverses notions de schémas sur $\F_1,$ il y a quelques traits communs. D'abord, on est tous d'accord sur le fait que la catégorie des anneaux commutatifs est très restrictive pour accueillir les nouveaux êtres, une limitation qui a été aussi mise en évidence, pour des propos différents, dans le contexte de la géométrie analytique globale \cite{Pa}. Autrement dit, $\Spec\Z$ est trop compliqué comme objet final de la catégorie des schémas affines. Ainsi, Durov \cite{Du} et Shai-Haran \cite{Ha} se proposent de l'élargir jusqu'à avoir $\Spec\F_1$, ou même un spectre encore plus simple, comme objet final. La tendance générale, c'est d'oublir la somme afin d'obtenir une géométrie non-additive. C'est le point de vue de Deitmar, qui imite la théorie des schémas juste en remplacent les anneaux par les monoïdes commutatifs. Dans son approche, un sous-monoïde $P$ de $M$ es premier si $xy\in P$ implique $x\in P$ ou $y\in P.$ Alors, un schéma affine est l'ensemble $\mathrm{Spec\ }M$ des sous-monoïdes premiers de $M$ muni d'un faisceau de monoïdes localisés (voir \cite{De1} pour les détails de cette construction).

\smallskip

Dans leur très beau article \textit{Au dessous de $\Spec\Z$} \cite{ToVa}, Töen et Vaquié ont développé la géométrie algébrique sur $\F_1$ comme un cas particulier de géométrie algébrique relative à une catégorie monoïdale symétrique vérifiant un certain nombre de propriétés techniques sur l'existence de limites. Si $(C,\times,\textbf{1})$ est une telle catégorie, on définit $\mathrm{Comm}(C)$ comme la catégorie de monoïdes commutatifs à l'intérieur de $C,$ ce qui va remplacer les anneaux commutatifs classiques. Alors, en posant $\mathrm{Aff}_C:=\mathrm{Comm}(C)^{\text{op}},$ avec $\mathrm{Spec}$ le foncteur d'une catégorie vers la catégorie duale, on obtient des schémas affines ayant une topologie de Grothendieck que l'on peut recoller pour avoir des schémas généraux. Les schémas sur $\F_1$ sont les schémas relatifs à la catégorie monoïdale symétrique $(\Ens,\times,\ast),$ où $\times$ est le produit cartésien d'ensembles. Lorsque l'on fait un certain relèvement, on obtient des schémas classiques, toujours une façon de tester que l'on manipule de bons candidats.  

\smallskip
Une approche plus terre à terre, mais qui permet d'étudier les fonctions zêta, est celle de Soulé \cite{So} raffinée ensuite dans les travaux de Connes et Consani \cite{CoCa,CoCa2}. De leur point de vue, un schéma sur $\F_1$ est défini par quelques données de descente d'un $\Z-$schéma de type fini $X_\Z.$ Une variété affine sur $\F_1$ est un triplet $(\underline{X},X_\C,e_X)$ formé d'un foncteur covariant $\underline{X}=\amalg_{k\geq 0}\underline{X}^k:\mathcal{F}_{ab}\longrightarrow \Ens$ de la catégorie des groupes abéliens finis vers la catégorie des ensembles gradués, d'une variété affine $X_\C$ et d'une transformation naturelle entre $\underline{X}$ et le foncteur $D\mapsto \Hom(\Spec\C[D], X_\C).$ À travers d'un caractère $\chi$ qui permet d'interpréter les éléments de $\underline{X}(D)$ comme des points de $X_\C,$ ils requièrent quelques hypothèses fortes garantissant que ces donnés définissent une seule variété affine sur $\Z.$ Très récemment, James Borger a proposée les $\lambda-$anneaux comme la structure codifiant ces données de descente de $\Z$ vers $\F_1$ \cite{Bo}.    

\smallskip

Dans ce mémoire, on aborde les constructions de Durov dans sa thèse \cite{Du}, dans laquelle il a développé tout un cadre algébrique permettant de surmonter les difficultés qui se présentent dès que l'on essaie d'appliquer la théorie d'Arakelov aux variétés non lisses, non propres ou ayant des métriques singulières. Outre qu'obtenir une nouvelle description de la géométrie d'Arakelov, qui pourrait servir à démontrer des analogues arithmétiques du théorème de Riemann-Roch, on propose un cadre suffisamment général pour traiter de la même manière la géométrie algébrique à la Grothendieck, la géométrie tropicale ou la géométrie sur le corps à un élément, juste en faisant varier \textit{l'anneau généralisé} de base sur lequel on travaille. Cela permet notamment de définir d'une façon rigoureuse $\F_1$ et ses extensions, non seulement les structures sur eux, ainsi que de construire la compactification de $\Spec\Z$ comme un pro-schéma généralisé projectif sur $\F_1.$  

\subsection*{Plan du mémoire}

On commence par exposer quelques limitations de l'analogie entre les corps de nombres et le corps de fonctions provenant du fait que $\Spec\Z$ ne soit pas un schéma propre et qu'il n'y ait pas un corps de base chez les entiers. Dans l'exposé, on rencontre trois objets à la recherche d'une définition satisfaisante: le corps à un élément $\F_1,$ l'anneau local à l'infini $\Z_\infty$ et la compactification de ${\Spec\Z}.$ C'est la tâche à laquelle on se consacre dans la section suivante, où l'on voit un anneau $R$ comme la monade sur les ensembles qui fait correspondre à chaque $X$ toutes les $R-$combinaisons linéaires formelles à support fini d'éléments dans $X.$ On extrait deux propriétés importantes de ce type de monades: l'algébricité, qui garantit essentiellement que l'on peut trouver une présentation par des opérations et des relations, et la commutativité, une des nouveautés introduites par Durov. On dit qu'une monade algébrique commutative est un anneau généralisé. Ensuite, on peut définir le spectre d'un anneau généralisé en calquant le cas classique. Dans la section suivante, on construit de nombreux exemples d'anneaux généralisés, parmi lesquels il faut souligner $\N,$ $\Z_\infty,$ $\Z_{(\infty)},$ ainsi que le corps à un élément $\F_1$ et ses extensions $\F_{1^n.}$ On obtient dans chaque cas des descriptions de la catégorie des modules sur ces anneaux. 

\smallskip

La seconde partie du mémoire présente la construction de $\widehat{\Spec\Z}.$ On définit d'abord les schémas de Zariski sur $\F_1$ en suivant le point de vue fonctoriel grothendieckien, d'après lequel un schéma sur $\F_1$ est grosso modo un foncteur de la catégorie des anneaux généralisés vers la catégorie des ensembles qui admet un recouvrement par des ouverts affines. Ensuite, on compactifie $\Spec\Z$ à la main à partir de deux anneaux généralisés $A_N$ et $B_N,$ et l'on obtient une description plus intrinsèque en termes de la construction $\mathrm{Proj}$ des monades graduées. Finalement, on fait le lien avec la géométrie d'Arakelov classique en discutant la compactification des variétés algébriques définies sur $\Q$ et le concept de fibré vectoriel sur $\widehat{\Spec\Z}$. On finit avec une discussion brève de quelques questions ouvertes. Puisque tout au long du travail il y a un usage intensif de la théorie des catégories, on a considéré opportun d'ajouter un annexe où l'on explique d'une façon concise les notions employés, sauf celles de catégorie, foncteur et transformation naturelle, que l'on suppose connues. Les termes définis dans l'annexe sont indiqués avec une petite étoile en haut. On présente de même une bibliographie exhaustive des articles concernant $\F_1.$    

\smallskip

\begin{acknowledgment}
Je tiens à remercier Frédéric Paugam, \textit{princeps functorum}, pour m'avoir proposé ce travail et pour les nombreuses heures qu'il a passé avec moi au tableau. J'ai discuté quelques idées du texte avec Javier López Peña (Queen Mary University of London), auquel je remercie l'envoie du préprint \cite{LoLo}, et avec Peter Arndt (Göttingen), lors de la conférence i-Math School on Derived Algebraic Geometry (Salamanca, juin 2009) où l'on s'est rencontrés par hasard. J'ai aussi bénéficié du cours d'Alain Connes au Collège de France sur la \textit{Thermodynamique des espaces non commutatifs}, où le contenu de \cite{CoCa2} a été présenté pour la première fois, et d'une conférence de Yuri Manin à l'Institut de Mathématiques de Jussieu \cite{Ma3}, dont il m'a envoyé les transparences très gentiment. Une partie de ce travail a fait objet de la conférence \textit{¿Es $\Z$ un anillo de polinomios?} au Séminaire de Géométrie Algébrique de l'Universidad Complutense de Madrid (séance du 26 mai) et de l'exposé court \textit{On Durov's compactification of $\Spec\Z$, the field with one element, the local ring at infinity and other rara avis} dans l'école doctorale en Géométrie Diophantienne tenue à l'Université de Rennes du 15 au 26 juin 2009.  
\end{acknowledgment}

%************************************************************************
\section{L'analogie entre les corps de nombres et les corps de fonctions}

Depuis les travaux constitutifs de plusieurs visionnaires qui ont rêvé d'une correspondance parfaite entre la géométrie et l'arithmétique (Kummer, Kronecker, Dedekind, Hensel, Hasse, Minkowski, Artin, Weil, etc.), l'analogie entre les corps de nombres et les corps de fonctions est l'une des idées les plus fructueuses de la mathématique moderne. Typiquement, des énoncés simples mais inaccessibles sur les corps de nombres (par exemple, la distribution des zéros des fonctions L ou les conjectures de Langlands) sont transformés en des théorèmes plus techniques dont la démonstration est possible grâce au nouveau cadre géométrique-analytique de la traduction dans les corps de fonctions (dans ce cas, les conjectures de Weil ou la correspondance de Langlands pour $GL_n$ établie par Lafforgue). Un exemple significatif, c'est le théorème de Mason, un exercice facile mettant en relation le degré de trois polynômes avec celui du radical de leur produit, dont l'analogue arithmétique, la \textit{conjecture abc}, a des conséquences qui vont du dernier théorème de Fermat en forme asymptotique et de quelques résultats d'infinitude des nombres premiers  jusqu'au théorème de Mordell-Faltings ou la conjecture de Lang \cite{Ni}.     

\smallskip

On rappel qu'un corps de nombres est une extension finie de $\Q$ et qu'un corps de fonctions est une extension finie $K$ de degré de transcendance 1 sur un corps $k,$ que l'on supposera en général fini ou algébriquement clos. Soit $K/k$ un corps de fonctions. À chaque valuation discrète de $K/k$ (i.e. une application $v:K\longrightarrow \Z\cup\{\infty\}$ telle que $v(xy)=v(x)+v(y),$  $v(x+y)\geq \min\{v(x),v(y)\}$ et que $v(x)=0$ quel que soit $x\in k),$ on associe l'anneau local   
\begin{equation*}
\mathcal{O}_v=\{x\in K: v(x)\geq 0\}. 
\end{equation*} Alors, l'ensemble $C_k$ d'anneaux de valuation discrète de $K/k$ est une courbe algébrique abstraite isomorphe à une courbe algébrique projective lisse $C$ ayant corps de fonctions $k(C)=K.$ Soit $\rho$ un nombre réel plus grand que 1. Étant donné une valuation discrète $v_p,$ si l'on pose $\abs{\cdot}_p=\rho^{-v_p(\cdot)},$ on obtient une norme sur $K.$ On a alors une bijection entre les points fermés de $C$ et les normes sur $K$ triviales sur $k$ à équivalence près (cf. \cite[I.6]{Har} pour les détails).

\smallskip

L'analogie entre les corps de nombres et les corps de fonctions devient plus forte lorsque l'on considère les entiers $\Z$ et les polynômes $k[T]$ sur un corps fini $k=\F_q,$ ainsi que leurs corps de fractions $\Q$ et $k(T).$ Ce sont des domaines euclidiens de dimension de Krull 1, dont les spectres sont formés par
\begin{align*}
&\mathrm{Spec\ }\Z=\{p\Z: \text{p premier}\}\cup \{0\}, \\ \mathbb{A}^1_k:=\mathrm{Spec\ }&k[T]=\{f\in k[T] \text{\ unitaire irréductible}\}.
\end{align*} De plus, si $k=\F_q,$ les corps résiduels $\F_p=\Z/p\Z$ et $\F_q[T]/(f)$ sont finis. D'autre part, si $f\in \mathbb{A}^1_k,$ l'anneau des séries de puissances $Z_f:=\varprojlim k[T]/(f^n)$ et le corps de séries de Laurent $Q_f:=Z_f[\frac{1}{f}]$ ont des équivalents $p-$adiques $\Z_p=\varprojlim \Z/p^n\Z$ et $\Q_p=\Z_p[\frac{1}{p}]$ ayant des plongements denses $\Z\subset \Z_p$ et $\Q\subset \Q_p.$ On peut même regarder les nombres comme fonctions sur $\mathrm{Spec\ }\Z,$ où l'image du premier $p$ par l'entier $n$ est la classe de congruence de $n$ modulo $p.$ Notons que l'espace d'arrivée de cette fonction dépend essentiellement du point où elle est évaluée.      

\smallskip

Quelles sont les limites de cette analogie? Le premier problème qui se pose, c'est l'existence d'un équivalent de la droite projective $\mathbb{P}^1_k:=\mathbb{A}^1_k \amalg_{\mathbb{G}_m}\mathbb{A}^1_k$ sur un corps de nombres. En effet, pour avoir la correspondance citée ci-dessus dans le cas plus simple $k(T)=k(\mathbb{P}^1_k),$ c'est nécessaire de passer au projective car l'on la norme
\begin{equation*}
\abs{\frac{f(T)}{g(T)}}=\rho^{\deg f-\deg g}
\end{equation*} correspondant au point de l'infini de $\mathbb{P}^1_k.$ Cela nous permet notamment d'avoir des formules produit comme
\begin{equation*}
\prod_{p\in\mathbb{P}^1_k, \ p\neq \xi}\abs{f}_p=1 \quad \forall f \in K^\ast,
\end{equation*} où $\xi$ est le point générique de $\mathbb{P}^1_k.$ Ces formules sont extrêmement importantes en géométrie diophantienne, où elles permettent, par exemple, de définir la hauteur d'un point d'un espace projectif en choisissant un représentant quelconque \cite{Bom}.  

\smallskip

En ce qui concerne les rationnels, les normes sur $\Q$ sont classifiées par le théorème d'Ostrowski, d'après lequel toute norme non-triviale est équivalente soit à une norme $p$-adique $\abs{\cdot}_p$ soit à la valeur absolue archimédienne usuelle. On rappel que $\abs{x}_p=p^{-v_p(x)},$ $v_p(x)$ étant le seul entier tel que $x=p^{v_p(x)}\frac{a}{b}$ avec $\pgcd(p,ab)=1,$ vérifie la propriété ultramétrique 
\begin{equation*}
\abs{x+y}_p\leq\max\{\abs{x}_p,\abs{y}_p\}. 
\end{equation*} On a de nouveau
\begin{equation*}
\prod_{(p)\in \Spec\Z, \ p\neq 0} \abs{x}_p=\frac{1}{\abs{x}_\infty} \quad \forall x \in \Q^\ast,
\end{equation*} ce qui suggère qu'il faut penser à une compactification $\widehat{\Spec\Z}$ pour avoir des résultats analogues aux théorèmes sur les corps de fonctions. C'est le point de vue de la théorie d'Arakelov comme développée par exemple dans \cite{La} ou \cite{SABK}. 

\smallskip

Cette brisure de la symétrie entre les places archimédiennes et non archimédiennes est mise en évidence par la nécessité d'introduire un facteur à l'infini afin d'obtenir l'équation fonctionnelle de la zêta de Riemann. En effet, selon l'interprétation de Tate \cite{Ta}, le produit d'Euler des facteurs locaux
\begin{equation*}
L_p(s)=\int_{\mathbb{Q}_p^\ast} \mathbbm{1}_{\mathbb{Z}_p}(x)\abs{x}_p^s \ d^\ast x=\frac{1}{1-\frac{1}{p^s}},
\end{equation*} où $d^\ast x=\frac{1}{1-\frac{1}{p^s}}\frac{dx}{\abs{x}_p}$ est la mesure de Haar sur $\mathbb{Q}_p,$ ne donne une équation fonctionnelle que lorsque l'on le complète avec un facteur à l'infini
\begin{equation*}
L_\infty(s)=\int_{\mathbb{R}^\ast} e^{-\pi x^2} \abs{x}_\infty^s \  \frac{dx}{\abs{x}},
\end{equation*} la gaussienne étant l'analogue mystérieux de la fonction indicatrice de $\Z_p$ dans $\Q_p.$    

\smallskip
Une autre difficulté qui se présente, c'est la construction des anneaux de valuation d'un corps de nombres. Pour $p$ premier, $\mathbb{Z}_p=\{x\in \mathbb{Q}_p: \abs{x}_p\leq 1\}$ est bien défini en vertu de la propriété ultramétrique, mais ce que jouerai le rôle d'anneau local à l'infini
\begin{equation*}
\mathcal{O}_{v_\infty}:=\mathbb{Z}_\infty=\{x \in \mathbb{R}=\mathbb{Q}_\infty: \abs{x}\leq 1\}=[-1,1] 
\end{equation*} n'est pas fermé sous l'addition. L'idée sera de la remplacer par des combinaisons linéaires convexes dans un cadre moins restrictif que l'algèbre commutative usuelle. Dans cette catégorie d'anneaux généralisés que l'on va construire dans la section suivante, $\Z_\infty$ aura le même status que les anneaux classiques.     

\smallskip
On peut attendre que cette façon d'\textit{envelopper les problèmes par le haut} jette une nouvelle lumière sur la manque de produits limitant l'analogie entre les corps de nombres et les corps de fonctions. Tandis qu'en géométrie on peut définir l'espace affine comme le produit de $n$ copies de la droite affine, ce qui revient à considérer le spectre du coproduit$^\star$ des $k-$algèbres $k[X_1,\ldots,X_n]=k[X]\otimes_k\ldots\otimes_k k[X],$ chaque tentative de définir une surface arithmétique ne donne pas de résultats. En effet, $\Z$ étant l'objet initial de la catégorie des anneaux, $\Spec \Z \times \Spec \Z$ est réduit à la diagonale. Cet obstacle est lié traditionnellement à l'absence d'un corps de base chez les entiers, sur lequel on pourrait prendre des produits $C \otimes_{\F_1} C$ comme celui dont Weil s'est servi dans sa démonstration de l'hypothèse de Riemann pour les courbes sur un corps fini. Dans la suite, on montrera que dans la catégorie des anneaux généralisés, l'élusive corps à un élément admet une définition naturelle. Malgré cela, le produit de deux copies de $\Spec \Z$ est encore trivial sur cette nouvelle catégorie.     

%************************************************************************

\section{La catégorie des anneaux généralisés}\label{int}
%************************************************************************

On rappel que la classe de toutes les catégories $\mathfrak{Cat}$ est elle même une $2-$catégorie, ce qui entraîne notamment que l'ensemble des foncteurs entre deux catégories quelconques est de nouveau une catégorie, dont les morphismes sont donnés par les transformations naturelles. Lorsque les deux catégories coïncident, la composition usuelle munit les endofoncteurs $\End(\mathcal{C}):=\Hom_\mathfrak{Cat}(\mathcal{C},\mathcal{C})$ d'une structure de catégorie monoïdale$^\star.$ Si $F$ et $G$ sont des endofoncteurs, on pose $FG:=F\otimes G$ et $F^n:=F\otimes\cdots\otimes F,$ ce qui est bien défini grâce à l'associativité. On appelle monade sur $\mathcal{C}$ toute algèbre dans $\End(\mathcal{C})$ qui soit compatible avec la structure de catégorie monoïdale. Plus explicitement:   

\begin{definition}
Une \textit{monade} sur une catégorie $\mathcal{C}$ est un triplet $\Sigma=(\Sigma, \mu, \epsilon)$ formé d'un endofoncteur $\Sigma: \mathcal{C} \longrightarrow \mathcal{C},$ d'une multiplication $\mu: \Sigma^2 \longrightarrow \Sigma$ et d'un morphisme identité $\epsilon: \id_\mathcal{C} \longrightarrow \Sigma$ tels que  
\begin{enumerate}
\item $\mu$ et $\epsilon$ sont des transformations naturelles, i.e. pour tout couple d'objets $X, Y$ dans $\mathcal{C}$ et tout morphisme $f: X \longrightarrow Y,$ les diagrammes suivants commutent
\begin{equation*}
\xymatrix{
& \Sigma^2(X) \ar[d]_{\Sigma^2(f)} \ar[r]^{\mu_X}
& \Sigma(X) \ar[d]^{\Sigma(f)} & &  X \ar[d]_f \ar[r]^{\epsilon_X} & \Sigma(X) \ar[d]^{\Sigma(f)} \\
& \Sigma^2(Y) \ar[r]^{\mu_Y} & \Sigma(Y) & &  Y  \ar[r]^{\epsilon_Y} & \Sigma(Y)} 
\end{equation*}
\item $\mu$ et $\epsilon$ respectent les axiomes d'associativité et d'unité, autrement dit, les diagrammes suivants sont commutatifs pour tout objet $X$ dans $\mathcal{C}$
\begin{equation*}
\xymatrix{
& \Sigma^3(X) \ar[r]^{\Sigma(\mu_X)} \ar[d]_{\mu_{\Sigma(X)}}
& \Sigma^2(X) \ar[d]^{\mu_X} & \id_\mathcal{C}\Sigma(X) \ar[r]^{\epsilon_{\Sigma(X)}} \ar[dr]_{\id_{\Sigma(X)}} & \Sigma^2(X) \ar[d]^{\mu_X} & \Sigma \id_\mathcal{C} (X) \ar[l]_{\Sigma(\epsilon_X)} \ar[dl]^{\id_{\Sigma(X)}} \\
& \Sigma^2 (X) \ar[r]^{\mu_X} & \Sigma(X) & &  \Sigma(X) & },
\end{equation*} 
\end{enumerate} Un \textit{morphisme de monades} $\varphi: \Sigma_1 \longrightarrow \Sigma_2$ est une transformation naturelle des endofoncteurs sous-jacents $\Sigma_1$ et $\Sigma_2$ telle que 
\begin{align}\label{morphmonad}
&\varphi_Y\circ\Sigma_1(Y)=\Sigma_2(f)\circ\varphi_X, \quad  \varphi_X\circ\epsilon_{\Sigma_{1, X}}=\epsilon_{\Sigma_{2,X}} \ \ \text{et que}\nonumber \\
\mu_{\Sigma_2,X}\circ&\Sigma_2(\varphi_X)\circ\varphi_{\Sigma_1(X)}=\varphi_X\circ\mu_{\Sigma_1,X}=
\mu_{\Sigma_2,X}\circ\varphi_{\Sigma_2(X)}\circ\Sigma(\varphi_X)
\end{align} pour tout couple d'objets $X,Y \in \text{Ob}(\mathcal{C})$ et tout morphisme $f:X \longrightarrow Y.$ On obtient ainsi la catégorie $\mathfrak{Monades}(\mathcal{C})$ des monades sur $\mathcal{C}.$ 
\end{definition}

\smallskip

\begin{exemple}
La catégorie $\mathfrak{Monades}(\Ens)$ admet un objet initial qui sera essentiel dans la suite. C'est le foncteur identité $\id_\Ens$ muni des applications $\mu$ et $\epsilon$ évidentes. On le note $\mathbb{F}_{\emptyset}$ et on l'appelle \textit{le corps sans éléments}. Les monades sur les ensembles possèdent aussi un objet final $\textbf{1}$ qui vaut $\{0\}$ sur tout ensemble $X$ et qui va jouer le rôle de l'anneau trivial dans la construction des anneaux généralisés. 
\end{exemple}

\smallskip

Étant donnée une monade $\Sigma=(\Sigma, \mu, \epsilon)$ sur $\mathcal{C},$ on peut lui associer de façon naturelle une notion de sous-structure et de module sur elle: 
\begin{definition}
Une \textit{sous-monade} $\Sigma'$ de $\Sigma$ est un sous-foncteur$^{\star}$ $\Sigma' \subset \Sigma$ stable par $\mu$ et $\epsilon,$ i.e. une monade $\Sigma'$ telle que
\begin{enumerate}
\item $\Sigma'(X) \subset \Sigma(X)$ pour tout objet $X$ dans $\mathcal{C}.$ 
\item la transformation naturelle $\Sigma'\longrightarrow \Sigma$ dont les composantes sont données par l'inclusion est un morphisme de monades.  
\end{enumerate}
\end{definition}

\smallskip

Si l'on se donne deux sous-monades $\Sigma_1$ et $\Sigma_2$ de la même monade $\Sigma,$ leur intersection est définie comme le produit fibré$^\star$ $\Sigma_1\times_\Sigma \Sigma_2$ dans la catégorie $\End(\mathcal{C}).$ Plus terre à terre, $\Sigma_1 \cap \Sigma_2$ est la monade dont l'image de chaque ensemble $X$ est l'intersection $\Sigma_1(X)\cap \Sigma_2(X)$ et dont les composantes des morphismes d'identité et de multiplication sont les restrictions de $\mu_X$ et $\epsilon_X$ à $\Sigma_1(X)\cap \Sigma_2(X).$ 

\smallskip

\begin{definition}
Un \textit{module} sur $\Sigma$ est la donné d'un couple $M=(M,\alpha)$ formé d'un objet $M \in \text{Ob}(\mathcal{C})$ et d'un morphisme $\alpha: \Sigma(M) \longrightarrow M$ tels que $\alpha\circ\mu_M=\alpha\circ\Sigma(\alpha)$ et que $\alpha \circ \epsilon_M=\id_M.$ Un \textit{morphisme de $\Sigma-$modules} $f: (M,\alpha_M) \longrightarrow (N,\alpha_N)$ est un morphisme $f \in \text{Hom}_\mathcal{C}(M,N)$ tel que $f\circ\alpha_M=\alpha_N\circ \Sigma(f).$ On construit ainsi la catégorie $\mathcal{C}^{\Sigma}$ des $\Sigma-$modules, qui peut être pensée plus intrinsèquement comme l'objet qui représente le foncteur $\Hom_\mathfrak{Cat}(\cdotp,\mathcal{C})^{\Sigma}$ dans $\mathfrak{Cat}.$ 
\end{definition}

\begin{definition} Soit $\varphi: \Sigma_1 \longrightarrow \Sigma_2$ un morphisme de monades sur une catégorie $\mathcal{C}.$ On appelle \textit{restriction des scalaires par rapport à $\varphi$} le foncteur $\varphi^\ast: \mathcal{C}^{\Sigma_2}\longrightarrow \mathcal{C}^{\Sigma_1}$ qui fait correspondre à un $\Sigma_2-$module $(N,\alpha)$ le $\Sigma_1-$module $(N,\alpha')$ ayant le même objet sous-jacent et dont l'action de $\Sigma_1$ est définie par $\alpha'=\alpha\circ\varphi_N.$ Lorsque ce foncteur admet un adjoint à gauche$^\star,$ on obtient aussi une notion d'\textit{extension des scalaires}.   
\end{definition}

\smallskip

Le foncteur d'oubli $\underline{O}: \mathcal{C}^{\Sigma} \longrightarrow \mathcal{C},$ qui associe à chaque $\Sigma-$module $(M,\alpha)$ l'objet sous-jacent $M,$ admet un adjoint à gauche, que l'on appellera comme d'habitude foncteur libre. C'est simplement $\underline{L}:\mathcal{C} \longrightarrow \mathcal{C}^{\Sigma}$ défini par la formule $\underline{L}(X)=(\Sigma(X),\mu_X)$ pour tout $X \in \text{Ob}(\mathcal{C}).$ La proposition suivante montre que ces deux foncteurs encodent toute la structure de la monade.

\begin{proposition}
Une monade $\Sigma$ sur $\mathcal{C}$ est entièrement déterminée par le couple de foncteurs adjoints $\underline{L}:\mathcal{C} \longrightarrow \mathcal{C}^{\Sigma}$ et $\underline{O}:\mathcal{C}^{\Sigma} \longrightarrow \mathcal{C}.$ 
\end{proposition}  
\begin{proof}
On se réfère à l'annexe pour la définition du produit de transformations naturelles. Si l'on considère l'unité et la co-unité de l'adjonction de $\underline{L}$ et $\underline{O},$ qui sont des transformations naturelles $\xi: \text{id}_\mathcal{C} \longrightarrow \underline{O}\underline{L}$ et $\eta: \underline{L}\underline{O} \longrightarrow \text{id}_{\mathcal{C}^\Sigma}$, on a
\begin{equation*}
\Sigma=\underline{O}\underline{L}, \quad \epsilon=\xi, \quad \mu=\underline{O}\star\eta\star\underline{L}.
\end{equation*} En effet, on a bien que $\mu: \underline{OLOL}=\Sigma^2 \longrightarrow \underline{OL}=\Sigma$ et pour tout objet $X$ dans $\mathcal{C}:$
\begin{equation*}
(\underline{O}\star\eta\star\underline{L})_X=\underline{O}(\eta_{\underline{L}(X)})=
\underline{O}(\eta_{(\Sigma(X),\mu_X)})=\underline{O}(\mu_X)=\mu_X. \qedhere
\end{equation*} 
\end{proof}

\smallskip

Cette construction admet une généralisation immédiate fournissant de nombreux exem-ples de monades. Supposons que $F: \mathcal{C} \longrightarrow \mathcal{D}$ et $G: \mathcal{D} \longrightarrow \mathcal{C}$ sont des foncteurs adjoints, avec des transformations $\xi: \text{id}_\mathcal{C} \longrightarrow GF$ et $\eta: FG \longrightarrow \text{id}_\mathcal{D}.$ Alors, si l'on pose $\Sigma=GF, \epsilon=\xi$ et $\mu=F\star\eta\star G,$ il est facile de vérifier que le triplet $(\Sigma,\epsilon,\mu)$ définit une monade sur $\mathcal{C}.$ Lorsque $\mathcal{C}=\mathfrak{Ens},$ $\mathcal{D}$ est une certaine catégorie d'ensembles munis de structure algébrique et $G: \mathcal{D} \longrightarrow \mathfrak{Ens}$ est le foncteur d'oubli, on obtiendra des monades dans la catégorie des ensembles, comme celle qui suit:   

\smallskip

\begin{exemple} [La monade des mots]
Soit $\mathcal{D}=\mathfrak{Mon}$ la catégorie dont les objets sont les monoïdes (i.e. les semi-groupes avec identité) et dont les morphismes sont les morphismes de semi-groupes qui préservent l'identité. Dans ce cas, le foncteur d'oubli vers les ensembles $\underline{O}:\mathfrak{Mon} \longrightarrow \mathfrak{Ens}$ admet comme adjoint à gauche le foncteur libre $\underline{L}:\mathfrak{Ens} \longrightarrow \mathfrak{Mon},$ qui associe à chaque ensemble $X$ le monoïde libre $L(X),$ ceci étant formé par les mots de longueur fini $x_1\ldots x_n,$ dans l'alphabet $X$ et ayant comme multiplication la concaténation des mots et comme identité le mot vide. On obtient ainsi une monade $M=(M, \mu, \epsilon)$ sur la catégorie des ensembles que l'on appellera la \textit{monade des mots}. On donne ensuite une description plus explicite. 

Si $X \in \text{Ob}(\mathfrak{Ens}),$ on identifie $X^0$ au mot vide et $X^n$ à l'ensemble des mots de longueur $n:$ $\{x_1\}\{x_2\}\ldots\{x_n\}.$ Alors $M(X)=\coprod_{n\geq 0}X^n.$ La correspondance entre $X$ et les mots de longueur un définit le morphisme d'unité $\epsilon_X(x)=\{x\},$ et le morphisme de multiplication $\mu_X: M^2(X) \longrightarrow M(X)$ agit de la façon suivante:
\begin{equation*}
\mu_X(\{\{x_1\}\ldots\{x_n\}\}\ldots\{\{z_1\}\ldots\{z_m\}\})=
\{x_1\}\ldots\{x_n\}\ldots\{z_1\}\ldots\{z_m\}
\end{equation*} La preuve que $M=(M,\mu,\epsilon)$ est en fait une monade se réduit donc à la vérification de l'associativité de l'opération "enlever les accolades". 
\end{exemple} 

\subsection{La monade associée à un anneau} \label{mona}
Étant donné un anneau $R,$ que l'on suppose unitaire mais pas forcément commutatif, l'adjonction du foncteur d'oubli $\underline{O}_R: R-\mathfrak{Mod} \longrightarrow \mathfrak{Ens}$ de la catégorie des $R-$modules vers les ensembles et du foncteur libre $\underline{L}_R: \mathfrak{Ens}\longrightarrow R-\mathfrak{Mod}$ définit une monade $\Sigma_R=(\Sigma_R,\mu,\epsilon)$ sur $\mathfrak{Ens}.$ Pour chaque $X,$ 
\begin{equation*}
\Sigma_R(X)=\text{Hom}_{\mathfrak{Ens}}^{\text{fini}}(X,\underline{O}_R(R))
\end{equation*}
est l'ensemble des applications de $X$ dans $R$ à support fini, qui s'identifie à l'ensemble des $R-$combinaisons linéaires formelles
\begin{equation*}
\Sigma_R(X)=\{\lambda_1\{x_1\}+\ldots+\lambda_n\{x_n\}: \lambda_i \in R, x_i \in X, n\geq 0\}.
\end{equation*} Ensuite, on définit le morphisme d'unité comme dans le cas de la monade des mots, i.e. $\epsilon_X(x)=\{x\}$ pour tout $x \in X,$ et la multiplication comme la transformation naturelle de $\Sigma_R^2$ dans $\Sigma_R$ dont les composantes calculent les combinaisons linéaires des combinaisons linéaires formelles. Autrement dit:
\begin{equation*}
\mu_X(\sum_i\lambda_i\{\sum_j \mu_{ij}\{x_j\}\})=\sum_{i,j}\lambda_i\mu_{ij}\{x_j\}
\end{equation*} 

\smallskip

Si $(X,\alpha)$ est un $\Sigma_R-$module, l'anneau $R$ agit sur l'ensemble $X$ au moyen des opérations $x+y:=\alpha(\{x\}+\{y\})$ et $\lambda x:=\alpha(\lambda\{x\}),$ qui vérifient bien les axiomes de $R-$module. Réciproquement, si l'on se donne un $R-$module $X,$ l'application qui évalue chaque combinaison linéaire formelle $\lambda_1\{x_1\}+\ldots+\lambda_n\{x_n\} \in \Sigma_R(X)$ en $\lambda_1 x_1+\ldots+\lambda_n x_n \in X$ définit une structure de $\Sigma_R-$module sur $X.$ Il est facile de montrer que cette correspondance est en fait une équivalence des catégories $R-\mathfrak{Mod}$ et $\mathfrak{Ens}^{\Sigma_R},$ ce qui revient à dire que le foncteur $\underline{O}_R$ est monadique. 
  
\smallskip

\begin{proposition}\label{prop2} Soient $R$ un anneau unitaire et $\Sigma_R$ la monade sur les ensembles construite comme ci-dessus. Alors:
\begin{enumerate}
\item Si $S$ est un sous-anneau de $R,$ alors $\Sigma_S$ est une sous-monade de $\Sigma_R.$
\item $\Sigma_R$ commute avec les limites inductives filtrées$^\star$. 
\item La structure de monoïde multiplicatif de $R$ peut être récupérée à partir de $\Sigma_R.$
\item L'application $R \mapsto \Sigma_R$ est un foncteur pleinement fidèle$^\star$ de la catégorie des anneaux unitaires dans la catégorie des monades. De plus, il préserve les suites exactes. 
\end{enumerate}
\end{proposition}

\begin{proof} La première assertion découle simplement du fait qu'un sous-anneau $S$ de $R$ contient l'unité et il est fermé pour le produit. Par conséquent, $\Sigma_S$ est compatible avec $\epsilon$ et $\mu,$ donc c'est une sous-monade de $\Sigma_R.$ Pour prouver la seconde affirmation, soit $X=\varinjlim_{I}X_i$ la limite du système  inductive filtré $(X_i)_{i\in I}.$ On a par définition $\Hom_\Ens(X,Z)=\varinjlim_{I}\Hom_\Ens(X_i,Z)$ pour tout ensemble $Z.$ Notamment, lorsque l'on prend $Z=\underline{O}_R(R)$ et on se restreint aux applications à support fini, on obtient $\Sigma_R(X)=\varinjlim_{I}\Sigma_R(X_i).$

Montrons dans ce qui reste la nature des relations entre un anneau $R$ et sa monade associée $\Sigma_R.$ Soit $\textbf{1}=\{1\}$ l'objet final de $\mathfrak{Ens}.$ Alors $\Sigma_R(\textbf{1})=\{\lambda\{1\}: \lambda \in R\}$ est l'ensemble sous-jacent à l'anneau, que l'on note $\abs{\Sigma_R}.$ On peut le voir comme un module à gauche sur lui même par l'égalité $R_s:=\underline{L}_R(\textbf{1}),$ ce qui nous permet de récupérer la multiplication de $R.$ Notons que l'unité est l'image $\epsilon_\textbf{1}(1)\in\abs{\Sigma_R}.$  

Maintenant, si l'on considère un morphisme d'anneaux $f: R \longrightarrow S,$ la famille d'applications ensemblistes $\Sigma_{f,X}: \Sigma_R(X) \longrightarrow \Sigma_R(S)$ définies par 
\begin{equation*}
\sum_{i=1}^n\lambda_i\{x_i\} \longmapsto \sum_{i=1}^n f(\lambda_i)\{x_i\} 
\end{equation*} induit un morphisme $\Sigma_f$ entre les monades $\Sigma_R$ et $\Sigma_S.$ En effet, la première condition de compatibilité dans (\ref{morphmonad}) se vérifie automatiquement, la seconde équivaut à dire que $f(e_R)=e_S$ et la dernière découle du fait que $f(x+y)=f(x)+f(y)$ et que $f(xy)=f(x)f(y).$ Donc, $\Sigma_f$ est un morphisme de monades. 

Pour prouver que ce foncteur est pleinement fidèle, il suffit de montrer que, étant donnés deux anneaux $R$ et $S,$ tous les morphismes de monades $\zeta: \Sigma_R \longrightarrow \Sigma_S$ sont de la forme $\zeta=\Sigma_f$ pour un unique morphisme d'anneaux $f: R \longrightarrow S.$ En effet, on sait déjà que si $\zeta: \Sigma_R \longrightarrow \Sigma_S$ est un morphisme de monades, $f:=\abs{\zeta}$ est un morphisme entre les monoïdes sous-jacents $R$ et $S.$ Il suffit de montrer que $\zeta: R \longrightarrow S$ respecte aussi l'addition et que $\Sigma_{f,X}=\zeta_X$ pour chaque ensemble $X$.      
\end{proof}

\begin{corollaire}
La catégorie des anneaux unitaires $\mathfrak{Ann}$ s'identifie à une sous-catégorie pleine$^\star$ de $\mathfrak{Monades}.$
\end{corollaire}

\begin{remarque}
Même si l'on vient de montrer que tout sous-anneau de $R$ fournit une sous-monade de $\Sigma_R,$ il n'y a pas une correspondance bijective entre les sous-anneaux de $R$ et les sous-monades de $\Sigma_R.$ Si l'on prend, par exemple, $R=\mathbb{Z},$ il n'existe pas de sous-anneaux non-isomorphes à $\Z$, mais par contre on peut construire la sous-monade $\N,$ qui n'est pas isomorphe à $\Z$ (voir \ref{sousm}).  
\end{remarque}

\begin{remarque} Si l'on calcule le foncteur de restriction des scalaires par rapport au morphisme $\Sigma_f$ introduit ci-dessus, on obtient le foncteur usuel de $S-\mathfrak{Mod}$ vers $R-\mathfrak{Mod}$ qui à chaque $S-$module $M$ fait correspondre le $R-$module $M$ dont l'action est donné par $r\cdot m=f(r)\cdot m.$ Par unicité de l'adjonction, le foncteur d'extension est dans ce cas l'application $M \longmapsto S \otimes_R M$ de $R-\mathfrak{Mod}$ vers $S-\mathfrak{Mod}.$     
\end{remarque}

\smallskip

En vue de la proposition \ref{prop2}, on définit l'ensemble sous-jacent à une monade $\Sigma$ dans la catégorie des ensembles comme $\abs{\Sigma}:=\Sigma(\textbf{1}).$ L'isomorphisme canonique 
\begin{equation*}
\varphi: \abs{\Sigma}=\Hom_{\mathfrak{Ens}}(\textbf{1},\Sigma(\textbf{1})) \cong \End_{\Sigma-\mathfrak{Mod}}(\underline{L}_R(\textbf{1}))
\end{equation*} munit $\abs{\Sigma}$ d'une structure de monoïde. En effet, étant donnés deux éléments $x$ et $y$ de $\abs{\Sigma},$ en posant $x\cdot y:=\varphi^{-1}(\varphi(x)\circ\varphi(y))$ on obtient une loi associative dont l'unité est $\epsilon_\textbf{1}(1).$ Il n'y a aucune raison pour attendre que cela soit un monoïde commutatif, car la composition d'endomorphismes ne l'est pas.    

\subsection{Les monades algébriques}

Pour définir une catégorie d'anneaux généralisés adéquate, il faudra imposer quelques restrictions aux monades. La première d'entre elles, l'algébricité, nous permettra d'obtenir une notion d'anneau généralisé non commutatif. 

\smallskip

\begin{definition}
Un endofoncteur $\Sigma: \mathfrak{Ens} \longrightarrow \mathfrak{Ens}$ est dit \textit{algébrique} s'il commute avec des limites inductives filtrées, c'est-à-dire, si $\Sigma(\varinjlim_I X_i)\cong \varinjlim_I \Sigma(X_i)$ pour chaque famille d'ensembles indexée par un ensemble partiellement ordonné filtré. Une monade est algébrique si son foncteur sous-jacent l'est. 
\end{definition}

\smallskip

\begin{exemple}[Une monade non algébrique] Comme on a montré dans la proposition \ref{prop2}, la monade associée à un anneau est algébrique. Si l'on pense à $\Sigma(X)$ comme l'ensemble des $\Sigma-$combinaisons linéaires formelles d'éléments de $X,$ il faut trouver des combinaisons à support infini pour avoir un exemple de monade non algébrique. Soit $\hat{\Z}_\infty$ (voir la section \ref{zinfty} pour le choix de cette notation) la monade définie par 
\begin{equation*}
\hat{\Z}_\infty(X):=\{\sum_{x\in X} \lambda_x: \ \lambda_x\in\R, \  \sum_{x\in X}\abs{\lambda_x}\leq 1\}.
\end{equation*} On montre que ce n'est pas une monade algébrique en la testant sur le système $X_\bullet=((X_n)_{n\in \N}, (i_{n,m})_{n\leq m})$ formé des ensembles $X_n=\{0,\ldots,n\}$ et des inclusions de $X_n$ dans $X_m$, la limite inductive étant $\N.$ D'un côté, $\hat{\Z}(\N)$ est l'ensemble des séries convergentes dont la somme des modules est plus petite ou égale à 1. De l'autre côté, $\hat{\Z}(X_n)$ est formé des $(n+1)$-uplets $(\lambda_0,\ldots,\lambda_n)$ telles que $\abs{\lambda_0}+\ldots+\abs{\lambda_n}\leq 1$ et la limite inductive est la réunion de tous ces ensembles. C'est clair que dans ce cas on ne peut pas échanger la limite car $\hat{\Z}(\N)-\varinjlim\hat{\Z}(X_n)$ contient toutes les séries convergentes avec un nombre infini de termes non nuls dont la somme des modules est plus petite ou égale que 1, par exemple $\sum_{n=0}^\infty \frac{1}{(n+1)^2\pi^2}$.      
\end{exemple}

\smallskip

Puisque la composition de deux foncteurs algébriques est encore algébrique, ils forment une sous-catégorie monoïdale de $\text{End}(\mathfrak{Ens}),$ qui est en plus pleine. On remarque aussi qu'un foncteur algébrique est complètement déterminé par sa valeur dans les ensembles finis, car tout ensemble $X$ est la limite inductive filtrée de ses sous-ensembles finis. Si $\Sigma$ est un endofoncteur algébrique et l'on note par $\underline{\mathbb{N}}$ la catégorie des ensembles finis standard $\textbf{n}=\{1,\ldots,n\}$ dont les morphismes sont les applications entre ensembles finis, la restriction $\Sigma \longmapsto \Sigma_{|\underline{\mathbb{N}}}$ induit une équivalence catégorique entre  $\mathfrak{Ens}^{\underline{\mathbb{N}}}:=
\text{Hom}_\mathfrak{Cat}(\underline{\mathbb{N}},\mathfrak{Ens})$ et les endofoncteurs algébriques. On peut les imaginer donc comme une collection d'ensembles $\{\Sigma(\textbf{n})\}_{n\geq 0}$ munie des applications fonctorielles $\Sigma(\varphi):\Sigma(\textbf{m})\longrightarrow \Sigma(\textbf{n})$ qui sont définies pour chaque $\varphi: \textbf{m} \longrightarrow \textbf{n}.$ 

\smallskip
Étant donné un foncteur $G: \underline{\mathbb{N}} \longrightarrow \mathfrak{Ens}$ et un ensemble quelconque $X,$ on pose:
\begin{equation*}
H_0(X):=\bigsqcup_{n\geq 0} G(\textbf{n})\times X^n, \quad H_1(X):=\bigsqcup_{\varphi:\textbf{m}\longrightarrow\textbf{n}} G(\textbf{m})\times X^n.
\end{equation*} On a alors des applications $p,q: H_1(X) \longrightarrow H_0(X)$ définies de la façon suivante: si $X^\varphi$ représente l'application $(x_1,\ldots,x_n) \mapsto (x_{\varphi(1)},\ldots,x_{\varphi(n)}),$ alors les restrictions de $p$ et $q$ à la composante indexée par $\varphi: \textbf{m} \longrightarrow \textbf{n}$ sont donnés par 
\begin{align*}
id_{G(\textbf{m})}\times X^\varphi&: G(\textbf{m})\times X^n \longrightarrow G(\textbf{m})\times X^m, \\
G(\varphi)\times \id_{X^n}&: G(\textbf{m})\times X^n \longrightarrow G(\textbf{m})\times X^n.  
\end{align*} Le lemme suivant, dont la démonstration se trouve dans \cite[4.1.4]{Du}, permet de calculer $\Sigma(X)$ comme le coégaliseur$^\star$ du système ci-dessus:   

\begin{lemme} Soit $\Sigma$ une monade algébrique et soit $G=\Sigma|_{\underline{\mathbb{N}}}\in \mathfrak{Ens}^{\underline{\N}}$ le foncteur image de $\Sigma$ par l'équivalence de catégories. Alors
\begin{equation*}
\Sigma(X)=\mathrm{Coeg}(p,q: H_1(X) \rightrightarrows H_0(X)) 
\end{equation*}
\end{lemme} 

\smallskip

Par conséquent, si $X$ et $Y$ sont des ensembles, une application $\alpha: \Sigma(X) \longrightarrow Y$ est la donnée d'une famille $\{\alpha^{(n)}: \Sigma(\textbf{n})\times X^n \longrightarrow Y \}_{n\geq 0}$ qui fait commuter le diagramme 
\begin{equation}\label{don}
\xymatrix{
& \Sigma(\textbf{m})\times X^n \ar[d]_{\Sigma(\varphi)\times id_{X^n}} \ar[rr]^{id_{\Sigma(\textbf{m})}\times X^{\varphi}} 
& & \Sigma(\textbf{m})\times X^m \ar[d]^{\alpha^{(m)}} \\
& \Sigma(\textbf{n})\times X^n \ar[rr]^{\alpha^{(n)}} & & Y}
\end{equation} pour tout $\varphi: \textbf{m} \longrightarrow \textbf{n}.$ Cela montre que $\Sigma(\textbf{n})$ paramètre en quelque mode les opérations n-aires de $X$ vers $Y$ et on peut écrire $[t]_\alpha(x_1,\ldots,x_n)$ au lieu de $\alpha(t,x_1,\ldots,x_n).$ On dit que $t \in \Sigma(\textbf{n})$ est une \textit{opération n-aire formelle} $[t]_\alpha: X^n \longrightarrow Y.$ Dans le nouveau symbolisme, la commutativité du diagramme ci-dessus se traduit en
\begin{equation*}
[\Sigma(\varphi)(t)]_\alpha(x_1,\ldots,x_n)=[t]_\alpha(x_{\varphi(1)},\ldots,x_{\varphi(m)})
\end{equation*} pour toute opération formelle $t \in \Sigma(\textbf{m})$ et pour toute application $\varphi: \textbf{m}\longrightarrow \textbf{n}.$

\medskip

D'après le lemme, il existe une surjection $H_0(X) \twoheadrightarrow \Sigma(X)$ dont la restriction à $\Sigma(\textbf{n})\times X^n$ applique $(t,x_1,\ldots,x_n)$ dans $\Sigma(x)(t)\in \Sigma(X),$ $x: \textbf{n} \longrightarrow X$ étant l'application $k \mapsto x_k.$ Notons $t(\{x_1\},\ldots,\{x_n\})$ cette image. Alors pour tout $s$ dans $\Sigma(X)$ il existe $n\leq 0, t\in\Sigma(\textbf{n})$ et $x_1,\ldots,x_n\in X$ tels que $s=t(\{x_1\},\ldots,\{x_n\}).$ En fait, $\Sigma(X)$ est l'ensemble de toutes ces expressions modulo les relations 
\begin{equation*}
(\Sigma(\varphi))(t)(\{x_1\},\ldots,\{x_n\})\sim t(\{x_{\varphi(1)}\},\ldots,\{x_{\varphi(n)}\}). 
\end{equation*}   

\smallskip

Ce point de vue nous permet de déterminer univoquement une monade algébrique à partir des données suivantes:

\smallskip

\begin{enumerate}
\item Une collection d'ensembles $\{\Sigma(\textbf{n})\}_{n\geq 0}$ et d'applications $\Sigma(\varphi): \Sigma(\textbf{m}) \longrightarrow \Sigma(\textbf{n})$ définies pour chaque $\varphi: \textbf{m} \longrightarrow \textbf{n}$ et telles que 
\begin{equation*}
\Sigma(id_{\textbf{n}})=id_{\Sigma(\textbf{n})}, \quad \Sigma(\psi\circ\varphi)=\Sigma(\psi)\circ\Sigma(\varphi).
\end{equation*} Puisque $\Sigma(X)=\varinjlim \Sigma(\textbf{n}),$ après avoir choisi d'identifications convenables entre les sous-ensembles finis de $X$ et les ensembles finis standard $\textbf{n}$, on peut reconstruire à partir de cette famille la valeur de la monade algébrique dans tous les ensembles. 

\smallskip

\item Une famille de morphismes de multiplication $\mu_n^{(k)}:\Sigma(\textbf{k}) \times \Sigma(\textbf{n})^k \longrightarrow \Sigma(\textbf{n})$ qui vérifient les relations
\begin{align*}
\mu_n^{(k)}\circ (id&_{\Sigma(\textbf{k})}\times \Sigma(\textbf{n})^\varphi )=\mu_n^{(k')}\circ (\Sigma(\varphi)\times id_{\Sigma(\textbf{n})^{k'}}) \\
&\mu_n^{(k)}\circ(id_{\Sigma(\textbf{k})}\times \Sigma(\psi)^\varphi)=\mu_m^{(k)}
\end{align*}
pour chaque $\varphi: \textbf{k} \longrightarrow \textbf{k}'$ et $\psi: \textbf{m} \longrightarrow \textbf{n}.$ En effet, d'après ce que l'on a vu précédemment, se donner le morphisme $\mu_X: \Sigma^2(X)\longrightarrow \Sigma(X),$ c'est équivalent lorsque l'on remplace $X$ et $Y$ par $\Sigma(X)$ dans (\ref{don}) à se donner une famille $\mu^{(k)}: \Sigma(\textbf{k})\times \Sigma(X)^k \longrightarrow \Sigma(X).$ Dans une seconde étape, on considère les restrictions à chaque $\Sigma(\textbf{n}),$ ayant ainsi   $\mu_n^{(k)}:\Sigma(\textbf{k}) \times \Sigma(\textbf{n})^k \longrightarrow \Sigma(\textbf{n}).$

\smallskip

\item Un élément distingué $\textbf{e}\in \Sigma(1).$ Comme $\id_{\mathfrak{Ens}}=\Hom_{\mathfrak{Ens}}(\textbf{1},\cdot),$ par le lemme de Yoneda$^\star$ le morphisme d'unité $\epsilon: \id_{\Ens}\longrightarrow \Sigma$ est déterminé univoquement par un élément $\textbf{e}:=\epsilon_{\textbf{1}}(\textbf{1})\in\Sigma(\textbf{1}).$ Si $\tilde{x}:\textbf{1}\longrightarrow X$ est la seule application ayant image $x,$ alors $\epsilon_X(x)=(\Sigma(\tilde{x}))(\textbf{e}).$  
\end{enumerate}

\smallskip
Ces données doivent vérifier de plus les axiomes de monade. Ce n'est pas difficile à voir que les conditions de la définition 1 se traduisent en:
\begin{enumerate}
\item $\mu_n^{(1)}(\textbf{e},t)=t$ pour tout $t \in \Sigma(\textbf{n})$ et pour tout $n\geq 0.$
\item $\mu_n^{(n)}(t,\{1\}_\textbf{n},\ldots,\{n\}_\textbf{n})=t$ pour tout $t\in \Sigma(\textbf{n}).$  
\item Pour tout $n,k,m\geq 0,$ le diagramme suivant commute:
\begin{equation}\label{monadalg}
\xymatrix{\Sigma(\textbf{k})\times \Sigma(\textbf{n})^k \times \Sigma(\textbf{m})^n \ar[rrr]^{\id_{\Sigma(\textbf{k})}\times(\mu_m^{(n)})^{(k)}} \ar[d]^{\mu_n^{(k)}\times\id_{\Sigma(\textbf{m})^n}}& & & \Sigma(\textbf{k})\times \Sigma(\textbf{m})^k \ar[d]^{\mu_m^{(k)}}\\
\Sigma(\textbf{n})\times \Sigma(\textbf{m})^n \ar[rrr]^{\mu_m^{(n)}}& & & \Sigma(\textbf{m})}
\end{equation}
\end{enumerate}

\smallskip

\begin{notation} Soit $\Sigma$ une monade algébrique définie par les données que l'on vient de décrire. Posons $||\Sigma||:=\bigsqcup_{n\geq 0} \Sigma(\textbf{n}).$ On appelle \textit{constantes} de $\Sigma$ les éléments de $\Sigma(\textbf{0}),$ \textit{opérations unaires} les éléments de $\Sigma(\textbf{1})$ et en général \textit{opérations n-aires} les éléments de $\Sigma(\textbf{n})$ pour $n \geq 2.$ Lorsque l'on veut souligner que $u$ est une opération $n-$aire, on écrit $u^{[n]}.$ Si $\Sigma(\textbf{0})$ n'a qu'un élément, on dit que $\Sigma$ est une monade avec zéro. Si $\lambda$ est une opération unaire, on écrit normalement $\lambda x$ au lieu de $\lambda_X(x);$ par exemple, $-x$ au lieu de $[-]_X(x).$ De même, si $+$ est une opération binaire, on n'écrit pas $[+]_X(x_1,x_2)$ mais $x_1+x_2.$    
\end{notation}

\subsection{Présentation d'une monade algébrique} 

Étant donnés une monade algébrique $\Sigma$ et une partie $U\subset ||\Sigma||,$ on considère la plus petite sous-monade contenant $U,$ ce qui revient à dire l'intersection de toutes les sous-monades $\Sigma_\alpha$ de $\Sigma$ dont leurs ensembles d'opérations contiennent $U.$ Notons-la
\begin{equation*}
\F_\emptyset \langle U \rangle:=\bigcap_{U\subset ||\Sigma_\alpha||}\Sigma_\alpha
\end{equation*} 

\begin{proposition} Soit $U$ une partie d'une monade algébrique $\Sigma.$ Alors $\Sigma':=\F_\emptyset\langle U \rangle$ est la sous-monade de $\Sigma$ dont l'ensemble $\Sigma'(\textbf{n})$ est formé des élément obtenus en appliquant un nombre fini de fois les règles suivantes:
\begin{enumerate}
\item $\{k\}_\textbf{n} \in \Sigma'(\textbf{n})$ pour tout $1\leq k \leq n.$
\item (Propriété de remplacement) Soit $u\in U$ une opération k-aire. Alors, pour tout $t_1,\ldots,t_k \in \Sigma'(\textbf{n}),$ on a   $[u]_{\Sigma(\textbf{n})}(t_1,\ldots,t_k) \in \Sigma'(\textbf{n}).$ 
\end{enumerate} 
\end{proposition}

\begin{proof} La preuve se réduit à la remarque que la collection de sous-ensembles $\Sigma''(\textbf{n})\subset \Sigma'(\textbf{n}) \subset \Sigma(\textbf{n})$ obtenus en appliquant ces deux règles un nombre fini de fois définit une sous-monade de $\Sigma$ contenant $U.$ Par conséquent $\Sigma''=\Sigma'$.  
\end{proof}

\smallskip

De plus, Durov montre que l'on peut toujours réaliser $\F_\emptyset\langle U \rangle$ comme une certaine sous-monade de la monade $M_U$ des mots à constants dans $U$ \cite[4.5.2]{Du}. On peut de même généraliser la construction ci-dessus de la façon suivante: soit $\Sigma_0$ une autre monade algébrique et soit $\rho: \Sigma_0 \longrightarrow \Sigma$ un morphisme de monades. Si $U$ est une partie de $||\Sigma||,$ on appelle sous-monade engendrée par $U$ sur $\Sigma_0$ 
\begin{equation*}
\Sigma_0\langle U \rangle=\F_{\emptyset}\langle U \cup ||\rho(\Sigma_0)||\rangle.
\end{equation*}

\smallskip

\begin{definition}
Soit $\Sigma$ une monade algébrique. On dit que $\Sigma$ est de \textit{type fini} sur $\Sigma_0$ s'il existe une partie finie $U$ de $||\Sigma||$ telle que $\Sigma=\Sigma_0\langle U \rangle.$ Une monade est dite absolument de type fini si elle de type fini sur $\F_\emptyset.$ 
\end{definition}

\smallskip

On peut construire de cette façon des monades libres, mais pour encoder l'information de la plupart des objets algébriques, on aura besoin d'un certain type de conditions de torsion, comme celles qui imposent, par exemple, les axiomes de la définition d'anneau. On le fait en définissant des relations d'équivalence algébrique compatibles avec la structure de monade:

\smallskip

\begin{definition}
Une relation $R\subset \Sigma \times \Sigma$ compatible avec la structure de monade de $\Sigma$ est la donnée d'un système de relations d'équivalence $\{R(\textbf{n})\subset \Sigma(\textbf{n})\times \Sigma(\textbf{n})\}_{n\geq 0}$ tel que si $t\equiv_{R(\textbf{k})} s, t_i \equiv_{R(\textbf{n})} s_i$ pour tout $1\leq i \leq k,$ alors: 
\begin{equation*}
[t]_{\Sigma(\textbf{n})}(t_1,\ldots,t_k)\equiv_{R(\textbf{n})}[s]_{\Sigma(\textbf{n})}(s_1,\ldots,s_k)
\end{equation*} En particulier, $R$ est compatible avec les applications induites $\Sigma(\varphi).$
\end{definition}

\smallskip

Soit $\Sigma$ une monade algébrique et $R$ une relation compatible avec $\Sigma.$ On définit alors le quotient $\Sigma'=\Sigma/R$ comme la monade dont les images des ensembles finis sont $\Sigma'(\textbf{n})=\Sigma(\textbf{n})/R(\textbf{n})$ et dont les morphismes de multiplication et d'unité sont les induits par passage au quotient. C'est la seule structure de monade dans $\Sigma'$ pour laquelle l'application évidente $\Sigma \longrightarrow \Sigma'$ est un morphisme de monades. Dans la suite, on appliquera cette construction en prenant comme relations les définies par un système d'équations, c'est-à-dire, par l'intersection $R=\cap_\alpha R_\alpha$ d'un certain nombre de relations algébriques compatibles. On note
\begin{equation*}
\Sigma_0\langle U | R\rangle:=\Sigma_0\langle U \rangle/R 
\end{equation*} la monade engendré sur $\Sigma_0$ par les opérations $U$ et par l'ensemble $R$ de relations.    

\smallskip

\begin{definition} Une monade algébrique $\Sigma$ admet une \textit{présentation finie} sur $\Sigma_0$ si elle de type fini et s'il existe un ensemble fini de relations $R$ tel que 
\begin{equation*}
\Sigma\cong \Sigma_0\langle U | R \rangle
\end{equation*}  
\end{definition}

\subsection{L'additivité}

Comme on a montré à la fin de l'épigraphe \ref{mona}, l'ensemble sous-jacent à une monade algébrique a toujours une structure de monoïde multiplicatif. On voudrait explorer maintenant la possibilité de définir une addition. Soit $\Sigma$ une monade algébrique ayant une constante $0\in\Sigma(\textbf{0}).$ On a des applications, dites de comparaison, $\pi_n: \Sigma(\textbf{n}) \longrightarrow \Sigma(\textbf{1})^n$ dont la composante $k-$ième applique
\begin{equation*}
t \mapsto t(0_{\Sigma(\textbf{1})},\ldots, \{1\}_\textbf{1} \ldots, 0_{\Sigma(\textbf{1})}).
\end{equation*} 

\smallskip

\begin{definition}
Soit $\Sigma$ une monade algébrique avec constante. Une opération binaire $[+]\in\Sigma(\textbf{2})$ est une \textit{pseudoaddition} si $\pi_2([+])=(\textbf{e},\textbf{e})\in\Sigma(1)^2.$ S'il s'agit de plus de la seule pseudoaddition que l'on peut définir sur $\Sigma,$ on dit que $[+]$ est une \textit{addition}. On appelle $\Sigma$ monade \textit{hypoadditive} (resp. hyperadditive, additive) si toutes les applications $\pi_n: \Sigma(\textbf{n}) \longrightarrow \Sigma(\textbf{1})^n$ sont injectives (resp. surjectives, bijectives). 
\end{definition}

\begin{exemple} Soit $\Sigma$ la monade définie par $\Sigma(\textbf{0})=\Z$ et par $\Sigma(\textbf{n})=\Z[T_1,\ldots,T_n]$ pour tout $n\geq 1$. Après avoir fixée $c=0$ comme constante, $\pi_2$ applique un polynôme $F(X,Y)$ dans $(F(T,0), F(0,T))\in \Sigma(\textbf{1})^2$. Cette application étant surjective, on conclut que $\Sigma$ est une monade hyperadditive. Elle n'est pas additive car il y a plusieurs façons de définir un polynôme à deux variables dont ses évaluations en $X=0$ et en $Y=0$ sont données. Par exemple, tous les polynômes de la forme $f=X.g+Y.h$, où $g(X,0)=h(0,Y)=1$, ont la même image par $\pi_2.$
\end{exemple}

\smallskip

La définition de pseudoaddition, qui peut sembler bizarre à première vue, ne dit que $x+0=x=0+x$ lorsque l'on la lit dans un $\Sigma-$module $X.$ Si $\Sigma$ est une monade additive, on obtient une addition en posant $[+]=\pi_2^{-1}(\textbf{e},\textbf{e}).$ Dans une monade hypoadditive, l'application $\pi_0: \Sigma(\textbf{0}) \longrightarrow \textbf{1}$ est injective, donc, il n'y a qu'un constante $0\in\Sigma(\textbf{0}).$ De plus, le fait que $\pi_2$ soit injective implique que s'il existe une pseudoaddition $[+],$ elle est nécessairement unique, c'est-à-dire, une addition. Alors, l'associativité et la commutativité de $[+]$ découlent automatiquement des identités pour $\pi_2(\{1\}+\{2\})$ et $\pi_3((\{1\}+\{2\})+\{3\})$.

\smallskip

\begin{proposition}
Les monades algébriques additives sont en correspondance bijective avec les semi-anneaux. De plus, une monade additive est la monade associée à un anneau s'il existe une symétrie $[-]$ (i.e. une opération unaire d'ordre deux). 
\end{proposition}

\begin{proof}
On rappel d'abord qu'un semi-anneau est un ensemble $R$ muni de deux opérations binaires $+$ et $\cdot$ telles que $(R,+)$ est un monoïde commutatif ayant élément neutre $0$ et que $(R,\cdot)$ est un monoïde avec unité $1.$ On demande de plus que la multiplication soit distributive par rapport à l'addition et que $0\cdot r=0=r\cdot 0$ pour tout $r\in R.$ Cette liste d'axiomes nous permet d'écrire toute opération n-aire $t\in\Sigma(\textbf{n})$ d'une façon unique sous la forme $\lambda_1\{1\}+\ldots+\lambda_n\{n\},$ ce qui revient à dire que l'application $\pi_n$ est bijective. Par conséquent, un semi-anneau définit une monade algébrique additive. 

Réciproquement, une monade algébrique additive est engendrée par la seule constante $0 \in \Sigma(\textbf{0})$, l'ensemble des opérations unaires $\abs{\Sigma}=\Sigma(\textbf{1})$ et l'addition $[+]\in \Sigma(\textbf{2})$. En utilisant la structure de monoïde multiplicatif de $\Sigma(\textbf{1})$, l'injectivité des applications $\pi_2$ et $\pi_3,$ ainsi que le fait que la monade n'admet qu'une seule constante, on obtient la présentation suivante:
\begin{align*}
\Sigma=\F_\emptyset\langle &0^{[0]}, \abs{\Sigma}, [+]^{[2]} | x+y=y+x, \ (x+y)+z=x+(y+z), \ x+0=x, \\
& \lambda(x+y)=\lambda x +\lambda y, \ (x+y)\lambda=x\lambda+y\lambda, \ \forall \lambda \in \abs{\Sigma} \rangle 
\end{align*} Alors, si l'on muni $\abs{\Sigma}$ du produit provenant de la structure de monoïde multiplicatif et de l'addition $+: \abs{\Sigma}^2 \longrightarrow \abs{\Sigma}$, les relations ci-dessus impliquent que $(\abs{\Sigma},+,\cdot)$ est un semi-anneau. C'est clair que les inverses pour l'addition sont la seule donnée supplémentaire dont on a besoin pour avoir en fait un anneau.   
\end{proof}

\begin{remarque}
Le nouveau cadre des monades algébriques permet de considérer des géométries relatives à des semi-anneaux tels que $\N$ ou $\R_{\geq 0}$ qui ne puissent pas être traitées avec les techniques usuelles de la géométrie algébrique. Cela a un spécial intérêt étant donnée que la géométrie tropicale, très en vogue depuis la fin des années 90, repose sur le semi-anneau $\mathbb{T}=(\R, \oplus, \otimes)$, où $x\oplus y:=\mathrm{min\ } \{x,y\}$ et $x\otimes y:=x+y$ (voir \cite{RST} pour une introduction). Puisque on n'as pas encore trouvé une définition fonctorielle des variétés tropicales, il serait intéressant d'étudier les propriétés de la géométrie relative à $\mathbb{T}$ au sens de Durov et voir si cela correspond en quelque sens à ce que l'on attend.      
\end{remarque}

%************************************************************************
\subsection{Les monades commutatives}
%************************************************************************

Malgré le succès du langage des monades algébriques pour encoder la géométrie sur de nouvelles structures, ce cadre est encore trop vaste pour avoir des équivalents des théorèmes classiques. De la même manière que l'on a introduit les monades algébriques en généralisant une propriété agréable des monades $\Sigma_R$ associées à un anneau, on va s'inspirer d'un autre trait des opérations définissant $\Sigma_R$ pour un anneau $R$ commutatif afin de trouver une sous-catégorie pleine des monades algébriques sur laquelle on travaillera dans la suite. En effet, dans un anneau classique $R=(R,+,\times)$, on a toujours le groupe abélien $(R,+)$ et on peut voir la distributivité comme une sorte de relation de commutation entre la somme et le produit. De plus, lorsque $R$ est commutatif, le monoïde sous-jacent $(R,\times)$ est aussi commutatif. Alors, dans un certain sens, un anneau usuel est commutatif si toutes les opérations commutent entre elles. Cela justifie la définition suivante:    

\smallskip

\begin{definition} Soit $\Sigma$ une monade algébrique. On dit que les opérations $t \in \Sigma(\textbf{n})$ et $s \in \Sigma(\textbf{m})$ \textit{commutent} si pour tout $\Sigma-$module $X$ et pour toute famille d'éléments $\{x_{ij}\}$ indexés par $1\leq i \leq n$ et par $1\leq j \leq m$ on a
\begin{equation*}
t(s(x_{11},\ldots,x_{1m}),\ldots,s(x_{n1},\ldots,x_{nm}))=s(t(x_{11},\ldots,x_{1n}),\ldots,t(x_{m1},\ldots,x_{mn}))
\end{equation*} Une monade algébrique est commutative si tout couple d'opérations dans $||\Sigma||$ commute.
\end{definition}

\smallskip

Cette définition, tout à fait effective, a un goût clairement matriciel. En effet, si l'on note par $M$ la matrice $n\times m$ formée des éléments $\{x_{ij}\}$, on peut faire agir les opérations $t$ et $s$ sur $M$ de deux  façons: premièrement, on évalue $s$ dans chaque ligne et on obtient $n$ valeurs auxquels on applique $t$; deuxièmement, on évalue $t$ dans chaque colonne et puis on applique $s$ aux $m$ éléments obtenus. Alors $t$ et $s$ sont commutatives si les deux procédés donnent le même résultat. Dans les exemples suivants on montre ce que cette égalité entraîne pour les opérations d'arité basse.   

\smallskip

\begin{exemple}[Commutativité des constantes]\label{const} Soient $c, d \in \Sigma(\textbf{0})$ deux constantes dans $\Sigma.$ La condition ci-dessus se lit $c=d$, donc une monade commutative admet au plus une constante. Cela permet de se référer sans ambiguïté à la constante $0$ d'une monade commutative. Soit maintenant $t$ une opération n-aire. Si $t$ et $c$ commutent, alors $t(c,\ldots,c)=c.$ Par conséquent, dans une monade algébrique $0$ reste invariant par toutes les opérations.   
\end{exemple}

\begin{exemple}[Commutativité des opérations unaires]\label{un} Soient $u, v \in \abs{\Sigma}$ deux opérations unaires. Si elles commutent, on a $uv=vu.$ On conclut que le monoïde sous-jacent à une monade commutative est aussi commutatif. Soit maintenant $t$ une opération n-aire. La commutativité de $u$ et $t$ se traduit en $t(ux_1,\ldots,ux_n)=ut(x_1,\ldots,x_n)$, qui devient la distributivité de la somme lorsque $t=[+]$.  
\end{exemple}

\begin{exemple}[Commutativité des opérations binaires]\label{bi} Soient $t, s \in \Sigma(\textbf{2})$ deux opérations binaires. On dit que $t$ et $s$ commutent si pour tout $x,y,z,t$ d'un $\Sigma-$module $X$ quelconque on a $t(s(x,y),s(z,t))=s(t(x,z),t(y,t))$. Lorsque $t=s=[+]$ cela signifie simplement que $(x+y)+(z+t)=(x+z)+(y+t)$. 
\end{exemple}

\smallskip

Il existe un rapport étroit entre l'additivité et la commutativité d'une monade algébrique. D'abord, c'est très facile à montrer qu'une monade commutative admet au plus une pseudoaddition (voir la fin de la paragraphe 4.6). Par conséquent, toute monade commutative hyperadditive est automatiquement additive. De plus, d'après la proposition 4, une monade additive est isomorphe à $\Sigma_{\abs{\Sigma}}$. On en déduit qu'une monade algébrique additive $\Sigma$ est commutative si et seulement si le semi-anneau $\abs{\Sigma}$ est commutatif. En vue de toutes ces considérations, les monades algébriques commutatives ressemblent suffisamment aux anneaux commutatifs pour introduire:  

\smallskip

\begin{definition} On appelle $\textit{anneau généralisé}$ une monade algébrique commutative sur les ensembles. On note par $\mathfrak{Gen}$ la sous-catégorie pleine des monades algébriques formé par les anneaux généralisés.  
\end{definition}

\smallskip

Une fois que l'on est arrivés à la définition d'anneau généralisé, on peut essayer de trouver la généralisation d'un corps. Évidemment, il ne suffit pas de demander que tous les éléments non nuls du monoïde commutatif sous-jacent soient inversibles. Entre les propriétés usuels des corps classiques, on choisit la non-existence d'idéaux propres comme définition. On dit qu'un anneau généralisé est sous-trivial s'il est isomorphe à une sous-monade de la monade finale $\textbf{1}$ décrite dans l'exemple 1. Ainsi,

\smallskip

\begin{definition}
Un anneau généralisé non sous-trivial $K$ est un \textit{corps généralisé} si tout quotient strict de $K$ différent de lui même est sous-trivial. 
\end{definition}   

\smallskip

Cette définition généralise la notion de corps au sens que tout corps classique est un corps généralisé. Comme on va montrer dans la section suivante, le corps à un élément et le corps résiduel de $\Z_\infty$ sont des corps généralisés. 

%************************************************************************
\section{Exemples}
%************************************************************************

Dans cette section, on montre à travers d'exemples d'anneaux généralisés la puissance du nouveau cadre algébrique, justifiant ainsi toute la machinerie qui a été développée afin de construire la catégorie des monades algébriques commutatives. On identifie toujours un anneau classique $R$ avec sa monade associée $\Sigma_R.$ Comme on avait promis dans la remarque 1, on commence par décrire la sous-monade de $\Z$ formée des combinaisons linéaires formelles à coefficients non négatifs. Ensuite, on construit l'anneau local à l'infini $\Z_\infty$ et on décrit les $\Z_\infty-$modules. Ces deux objets étant définis, on est prêts pour identifier le corps à un élément $\F_1$ avec leur intersection. On étudie de même les extensions cyclotomiques de $\F_1$ et l'on montre que $\Z$ admet une présentation finie sur $\F_1.$ Dans la partie finale, on définit, pour tout entier $N \geq 1,$ deux anneaux généralisés $A_N$ et $B_N$ qui sont à la base de la construction de la compactification de $\Spec \Z.$ 

\subsection{$\N$}\label{sousm}

Soit $\N$ l'endofoncteur des ensembles dont l'image de chaque $X$ est donnée par les combinaisons linéaires à support fini d'éléments de $X$ à coefficients non-négatifs: 
\begin{equation*}
\N(X)=\{\lambda_1\{x_1\}+\ldots+\lambda_n\{x_n\}: \ \lambda_i \in \mathbb{Z}, \lambda_i \geq 0, x_i \in X, n \geq 0 \}.
\end{equation*} Puisque $\epsilon(x)=\{x\}\in\N$ et le produit de deux entiers non-négatifs est encore non-négatif, $\N$ est compatible avec la multiplication et l'unité de $\Z,$ donc il s'agit d'une sous-monade dont l'ensemble sous-jacent est le monoïde des entiers non-négatifs.

\smallskip

Considérons un $\N-$module: c'est la donnée d'un ensemble $X$ et d'un morphisme $\alpha: \N(X)\longrightarrow X$ tel que $\alpha\circ\mu_X=\alpha\circ\N(\alpha)$ et que $\alpha\circ\epsilon_X=\id_X.$ Alors, l'opération $x\cdot y=\alpha(\{x\}+\{y\})$ munit $X$ d'une structure de monoïde commutatif et l'on montre comme dans (\ref{mona}) que la catégorie des $\N-$modules est en fait équivalente à celle des monoïdes commutatifs. La restriction de scalaires par rapport à l'inclusion $\N \longrightarrow \Z$ est donc le foncteur d'oubli des groupes abéliens vers les monoïdes commutatifs, qui admet comme adjoint le foncteur libre $D \longmapsto G[D]$.  

\subsection{$\mathbb{Z}_\infty$}\label{zinfty} Soit maintenant $\Sigma_{\infty}$ la sous-monade de $\R$ définie par les combinaisons octahedrales, c'est-à-dire, par les combinaisons linéaires formelles à support fini $\Sigma_i \lambda_i\{x_i\},$ où les $\lambda_i \in \R$ satisfont $\sum_i \abs{\lambda_i}\leq 1.$ En effet, c'est une sous-monade de $\R$ car $\{x\}\in \Sigma_\infty$ pour tout élément de la base et l'on a l'inégalité
\begin{equation*}
\sum_{i,j}\abs{\lambda_i \mu_{ij}}=\sum_i\abs{\lambda_i}\sum_j\abs{\mu_{ij}}\leq \sum_i \abs{\lambda_i}\leq 1, 
\end{equation*} ce qui revient à dire que les combinaisons formelles des éléments dans $\Z_\infty$ appartiennent de nouveau à $\Z_\infty.$ Le monoïde sous-jacent est $\abs{\Z_\infty}=[-1,1]$ avec la multiplication induite par celle de $\R.$ En remplacent les nombres réels par les complexes, on obtient la sous-monade $\overline{\Z}_\infty\subset \C.$  

\smallskip

Quels sont les $\Z_\infty-$modules? Par définition, un $Z_\infty-$module est un ensemble $X$ muni d'une application $\alpha: \Sigma_\infty(X) \longrightarrow X$ telle que $\alpha(\{x\})=x$ pour tout $x\in X$ et que
\begin{equation*}
\alpha(\sum_i \lambda_i \{\alpha(\sum_i \mu_{ij} \{x_j\})\})=\alpha(\sum_{i,j}\lambda_i\mu_{ij} \{x_j\})
\end{equation*} Remarquons que lorsque $X$ est un ensemble à $n$ points dans le plan affine, $\Sigma_\infty(X)$ s'identifie à l'enveloppe convexe des points. Cela nous amène à étudier les réseaux et les corps convexes dans un espace vectoriel. Afin de donner une description plus parlant de $\Z_\infty$-module, on construit d'abord la catégorie de $\Z_\infty-$réseaux, que l'on plonge ensuite dans la catégorie de $\Z_\infty$-modules sans torsion.  

\smallskip

Commençons par rappeler le cas $p$-adique. Fixons donc un nombre premier $p$ et soit $E$ un $\Q_p$-espace vectoriel de dimension finie. Alors, il existe une bijection entre l'ensemble de $\Z_p$-réseaux modulo l'action multiplicative de $\Q_p^\ast$ et les sous-monoïdes maximaux compactes (pour la topologie $p$-adique) de $\End(E)$, la correspondance étant donné par
\begin{equation*}
A \longmapsto M_A=\{g \in \End(E): g(A)\subset A\}
\end{equation*} Si $E$ est un espace vectoriel réel de dimension finie, on peut associer à toute forme quadratique définie positive $Q$ le sous-monoïde compact maximal
\begin{equation*}
M_Q=\{g\in \End(E): Q(g(x))\leq Q(x) \ \forall x \in E\}.
\end{equation*} Cependant, il existe des sous-monoïdes compacts maximaux n'ayant pas cette structure. Puisque $Q$ définit la norme $||x||=\sqrt{Q(x)}$, on peut remplacer les espaces vectoriels quadratiques par le cadre plus général des espaces vectoriels normés. En fait, on obtient ainsi une équivalence entre les normes sur $E$ et les corps compacts symétriques convexes inclus dans $E$ modulo l'action multiplicative de $R^\ast$. En effet, si l'on se donne une norme$ ||\cdot||$ sur $E$, l'ensemble $A_{||\cdot||}=\{x\in E: ||x||\leq 1\}$ est un corps compact symétrique convexe et, réciproquement, tout corps $A$ vérifiant ces propriétés définit la norme $||x||=\inf \{\lambda>0: \lambda^{-1}x \in A\}$. D'après \cite[2.3.3]{Du}, les corps compacts symétriques convexes sont en bijection avec le sous-monoïdes maximaux compacts (pour la topologie euclidienne usuelle) de $\End(E)$. Par analogie avec le cas $p$-adique, on dit que:

\begin{definition}
Un $\Z_\infty$-réseau est la donnée d'un couple $A=(A_{\Z_\infty}, A_\R)$ formé d'un espace vectoriel $A_\R$ de dimension finie et d'un corps compact symétrique convexe $A_{\Z_\infty}\subset A_\R$. Un morphisme de $\Z_\infty$-réseaux est un couple $f=(f_{\Z_\infty}, f_\R)$ formé d'une application linéaire $f_\R: A_\R \longrightarrow B_\R$ et d'une application $f_{\Z_\infty}: A_{\Z_\infty} \longrightarrow B_{\Z_\infty}$ telle que $f_{\Z_\infty}=f_{{\R|A}_{\Z_\infty}}$. 
\end{definition}   

Ensuite, on peut élargir la catégorie en oubliant le fait que l'espace vectoriel $A_\R$ soit de dimension finie et que le corps symétrique convexe soit compact, même fermé. On appelle $\Z_\infty$-modules plats (ou sans torsion) les nouveaux couples. Maintenant on utilise la construction généralisant la proposition 1 pour obtenir une notion de $\Z_\infty$-module. Puisque le foncteur d'oubli de la catégorie des $\Z_\infty$-modules sans torsion vers les ensembles 
\begin{equation*}
\Gamma_{\Z_\infty}: \Z_\infty-\mathfrak{Modpl} \longrightarrow \Ens, \quad A \longmapsto A_{\Z_\infty}=\Hom_{\Z_\infty}(\Z_\infty, A) 
\end{equation*} admet comme adjoint à gauche le foncteur libre 
\begin{equation*}
L_{\Z_\infty}: \Ens \longrightarrow \Z_\infty-\mathfrak{Modpl}, \quad X \longmapsto \Z_\infty(X),
\end{equation*} on obtient une monade $\Sigma_\infty$ sur les ensembles. À différence de ce qui se passait avec les monades associées à des anneaux, dans ce cas la catégorie $\Z_\infty-\mathfrak{Modpl}$ n'est pas équivalente à $\Ens^{\Sigma_\infty}$, mais une sous-catégorie pleine. On définit alors:

\begin{definition} Un $\Z_\infty$-module est un objet de la catégorie $\Ens^{\Sigma_\infty}$.
\end{definition} 

\smallskip

On rappel que dans la théorie standard des anneaux commutatifs, la localisation de $\Z$ dans le nombre premier $p,$ que l'on note $\Z_{(p)},$ est l'ensemble des fractions $\frac{m}{n}$ tels que $p$ ne divise pas $n.$ Cela correspond à considérer les éléments de $\Z_p$ qui appartient à $\Q,$ c'est-à-dire, l'intersection $\Z_{(p)}=\Z_p\cap \Q.$ Par analogie, on définit 
\begin{equation*}
\Z_{(\infty)}=\Z_\infty \cap \Q.
\end{equation*} Alors, $\Z_{(\infty)}(X)$ est l'ensemble des combinaisons linéaires octahedrales à coefficients rationnels. Dans la construction de $\widehat{\Spec\Z},$ on aura besoin d'exiger que les dénominateurs des coefficients soient seulement des puissances d'un entier $N.$ Cela revient à considérer l'intersection $A_N:=\Z[\frac{1}{N}]\cap \Z_{(\infty)}.$  

\subsection{Anneaux de valuation}

L'exemple précédent suggère la nécessité d'étudier les anneaux de valuation dans toute leur généralité. Soit $K$ un corps et soit $\abs{\cdot}_v$ une valuation sur $K$ archimédienne ou non. On pose $N_v:=\{x\in K: \abs{x}_v\leq 1\}$ et on définit l'anneau de valuation de $\abs{\cdot}_v$ comme étant la plus grande sous-monade algébrique $\mathcal{O}_v$ de $K$ telle que $N_v$ soit un $\mathcal{O}_v-$module. Alors, $\mathcal{O}_v$ est un anneau classique si et seulement si $\abs{\cdot}_v$ est une valuation archimédienne. En effet, 
\begin{equation*}
\mathcal{O}_v(\textbf{n})=\{(\lambda_1,\ldots,\lambda_n)\in K^n: \ \abs{\sum_{i=1}^n \lambda_i x_i}_v \leq 1 \ \text{si} \ \abs{x_i}_v\leq 1 \ \forall i=1,\ldots,n\} 
\end{equation*} Lorsque $\abs{\cdot}_v$ est archimédienne, la condition ci-dessus équivaut à $\abs{\lambda_i}_v\leq 1$ pour tout $i=1,\ldots,n,$ de sorte que $\mathcal{O}_v(\textbf{n})=N_v^n$. Si par contre $\abs{\cdot}_v$ n'est pas archimédienne, $\mathcal{O}_v(\textbf{n})$ dévient l'ensemble des n-uplets telles que $\abs{\lambda_1}_v+\ldots+\abs{\lambda_n}_v \leq 1$. En prenant la valeur absolue archimédienne dans $\Q,\R$ et $\C$ on obtient $\Z_{(\infty)}, \Z_\infty$ et $\overline{\Z}_\infty$ respectivement comme anneaux de valuation. On montre comme dans l'exemple 4.7 que l'on récupère les anneaux initiaux en localisant les anneaux de valuations dans le système multiplicatif engendré par un élément $0<\abs{f}<1$ quelconque. 

\smallskip

Si $\abs{\cdot}_v$ est une valuation archimédienne, on sait que $\mathcal{O}_v$ est un anneau local d'idéal maximal $\mathfrak{m}_v=\{x\in K: \ \abs{x}_v <1\}$ et l'on appelle corps résiduel le quotient de $\mathcal{O}_v$ par $\mathfrak{m}_v$. Étudions le cas non-archimédien sur $\R$. Cela revient à quotienter $\Z_\infty$ par les points intérieurs de $[-1,1],$ qui sont ceux à valeur absolue strictement plus petit que $1.$ Il s'agit d'un idéal de l'anneau généralisé $\Z_\infty$ au sens que l'on va introduire dans la section suivante. Comme ensemble sous-jacent $\F_\infty:=\Z_\infty/\mathfrak{m}_\infty$ est réduit à $\{-1,0,1\}$, qui admet la structure de $\Z_\infty$-module suivante:
\begin{displaymath}
\alpha(\lambda_{-1}{-1}+\lambda_0\{0\}+\lambda_1\{1\})= \left\{ 
\begin{array}{lr}
1 & \lambda_1-\lambda_{-1}=1 \\
-1 & \lambda_1-\lambda_{-1}=-1 \\
0 & \abs{\lambda_1-\lambda_{-1}}<1 
\end{array} \right. 
\end{displaymath}

Soit maintenant $x=\lambda_1\{1\}+\ldots+\lambda_n\{n\}$ un élément de $\Z_\infty(\textbf{n})$. Si $\sum_{i=1}^n \abs{\lambda_i}<1$, alors $[x]=[0]$ dans $\F_\infty(\textbf{n})$. Considérons $x$ et $y$ tels que $\sum_{i=1}^n \abs{\lambda_i}=1$. Dans ce cas, $[x]=[y]$ si et seulement si les suites des signes des coefficients $\lambda_i$ coïncident. Cette monade n'est pas additive, car tous les éléments différents de $\pm \{i\}$, $i=1,\ldots,n$ sont dans le noyau de l'application de comparaison $\pi_n: \F_\infty(\textbf{n}) \longrightarrow \abs{\F_\infty}^n$. Notamment, $\F_\infty$ est engendré par une constante $0^{[0]}$, une opération unaire $[-1]^{[1]}$ et une opération binaire $\ast^{[2]}$ défini dans chaque $\F_\infty$-module par $[x]\ast [y]=[(1-\lambda)x+\lambda y]$, $0<\lambda<1$ étant un nombre réel quelconque. L'opération $\ast$ vérifie les propriétés suivantes:
\begin{align*}
&0\ast 0=0, \quad x\ast 0=0, \quad x\ast x=x, \quad x \ast (-x)=0, \\
&(-x)\ast(-y)=-(x\ast y), \quad (x\ast y)\ast z=x\ast(y\ast z)  
\end{align*} Montrons, par exemple, la commutativité:
\begin{equation*}
x\ast y -y\ast x=(1-\lambda)x+\lambda y -(1-\lambda)y-\lambda x=(1-2\lambda)x+(2\lambda-1)y.
\end{equation*} Puisque $\abs{1-2\lambda}+\abs{2\lambda-1}<1$, l'élément $x\ast y-y\ast x$ est nul dans le quotient. Cela nous permet d'écrire la présentation suivante, où l'on utilise une extension de $\F_1$ que l'on va définir dans le paragraphe suivant:      
\begin{equation}\label{res}
\F_\infty=\F_{1^2}[\ast^{[2]}| \textbf{e}\ast\textbf{e}=0, \  \textbf{e}\ast\textbf{e}=\textbf{e}, \ x\ast y=y \ast x, \ (x\ast y)\ast z=x\ast(y\ast z)]
\end{equation} Remarquons que les relations découlant de la commutativité ont été omises. 

\subsection{$\mathbb{F}_1:$ le corps à un élément}

On peut finalement définir le corps à un élément comme la monade 
\begin{equation*}
\F_1:=\N\cap \Z_{\infty}.
\end{equation*} Quand on regarde l'image d'un ensemble $X$ par $\F_1,$ la condition de positivité implique que les seules combinaison linéaires permises sont $\{x\}$ pour quelque $x \in X,$ ou celle où tous les coefficients sont nuls. Autrement dit: $\F_1(X)=X\sqcup 0.$ 

\smallskip

Un $\F_1-$module est donc un couple $(X,\alpha)$ où $\alpha: X\sqcup 0 \longrightarrow X$ coïncide avec l'identité sur $X$ et $\alpha(0)=0_X \in X.$ Par conséquent, on identifie les $\F_1-$modules aux ensembles ayant un point distingué. Pour obtenir une présentation de $\F_1,$ il suffit d'ajouter à la monade initiale la constante $0\in\F_1(\textbf{0}),$ de sorte que $\F_1=\F_\emptyset\langle 0^{[0]} \rangle$ est la monade libre engendrée par une constante. Cela entraîne que se donner une monade algébrique ayant une constante $0\in\Sigma(\textbf{0})$ soit équivalent à se donner un morphisme de monades $\rho: \F_1 \longrightarrow \Sigma.$ De la présentation ci-dessus, on déduit que $\F_1$ est un corps généralisé, car le seul quotient strict, $\F_\emptyset$, est sous-trivial. 

\smallskip

Puisque le monoïde sous-jacent à $\F_1$ est $\{0,1\},$ le seule idéal propre est $\{0\},$ autrement dit, $\Spec \F_1$ est réduit à un point. C'est pour cela que l'on se réfère au spectre du corps à un élément comme le point absolu. On obtient un résultat analogue pour $\F_\emptyset$, ainsi que pour les extensions cyclotomiques de l'exemple suivant. Cependant, il y a encore des géométries au-dessus de $\F_1$. En effet, dans \cite{ToVa} les auteurs montrent qu'il y a un foncteur d'extension de base des schémas relatifs à la catégorie monoïdale $(\mathfrak{SEns},\times,\ast)$ des ensembles simpliciaux$^\star$ munis du produit direct vers les schémas sur $\F_1$. Le spectre de $\F_\emptyset$ étant l'objet initial de la catégorie des schémas affines généralisés, il serait intéressant d'élucider le rôle de cette géométrie dans l'approche de Durov.  

\subsection{Extensions cyclotomiques}

De même, l'extension $\F_{1^2}$ de $\F_1$ est définie comme la monade 
\begin{equation*}
\F_{1^2}:=\Z\cap \Z_{(\infty)}=\Z \cap \Z_\infty.
\end{equation*}
Ainsi, $F_{1^2}(X)$ est formé des combinaisons formelles à au plus un coefficient entier de valeur absolue $1,$ ce qui donne $0,$ les éléments de la base $\{x\}$ et leurs opposés $-\{x\},$ de sorte que $\F_{1^2}(X)=X \sqcup -X \sqcup 0.$ On a $\F_1=\N\cap \F_{1^2}.$ Évidemment, $\F_{1^2}$ est de type fini sur $\F_1,$ engendré par une opération unaire:
\begin{equation*}
\F_{1^2}=\F_\emptyset \langle 0^{[0]}, [-]^{[1]} \ | \ -(-x)=x \rangle  = \F_1 \langle [-]^{[1]} \ | \ -(-x)=x \rangle
\end{equation*}
Aux données qui définissent un $\F_1-$module, il faut ajouter l'action de $\alpha$ sur $-X$ afin d'avoir un module sur $\F_{1^2}.$ Puisque on peut la voir comme une involution laissant invariant $0_X,$ $\F_{1^2}-$Mod est la catégorie des ensembles avec un point distingué munis d'une involution qui préserve ce point. Cette extension est d'une spéciale importance depuis les travaux d'Alain Connes et Caterina Consani, qui ont montré dans \cite{CoCa} que les groupes de Chevalley n'ont pas une structure de variété sur $\F_1$ mais sur $\F_{1^2},$ en répondant de cette façon inattendue à la question posée par Soulé dans \cite{So}.  

\smallskip

Est-ce qu'il y a d'extensions de $\F_1$ de degré plus grand que deux? Même si cela pourrait paraître paradoxale, la réponse à cette question a été trouvée plus d'une dizaine d'ans avant d'avoir une définition précise de $\F_1.$ En effet, dans le manuscrit non publié \cite{KaSm} Kapranov et Smirnov ont proposé de penser à $\F_{1^n}$ comme le monoïde formé de zéro et les racines n-ièmes de l'unité $\mu_n$. En suivant des idées de Weil et d'Iwasawa selon lesquelles ajouter des racines de l'unité est équivalent à faire une extension du corps de base, ils suggèrent qu'un schéma $X$ est défini sur $\F_{1^n}$ lorsque l'anneau des fonctions régulières sur $X$ contient les racines n-ièmes de l'unité. On peut de même penser à la clôture algébrique de $\F_1$ ou à la droite affine $\mathbb{A}^1_{\F_1}$ comme le monoïde contenant zéro et toutes les racines de l'unité. Ces idées, ainsi que le travail \cite{Hab} de Habiro, ont inspiré les définitions de coordonnées cyclotomiques et de fonctions analytiques sur $\F_1$ \cite{Ma3}.         

\smallskip

Dans le langage des monades, $\F_{1^n}=\F_1\langle \zeta_n^{[1]} | \zeta_n^nx=x \rangle$. Remarquons que $\F_1$ étant un quotient strict de $\F_{1^n},$ il ne s'agit pas d'un corps généralisé si $n\geq 2$. Un $\F_{1^n}-$module est un ensemble muni d'une action libre de $\Z/n\Z$. La correspondance $\varphi_{m,n}: \zeta_n \longmapsto \zeta_{nm}^m$ induit un plongement de $\F_{1^n}$ dans $\F_{1^{nm}}$ pour des entiers $n,m\geq 1$ quelconques. Si l'on prend la limite inductive par rapport à $\varphi_{m,n}$, on obtient l'anneau généralisé
\begin{equation*}
\F_{1^\infty}:=\varinjlim \F_{1^n}=\F_1\langle \zeta_1,\zeta_2,\ldots | \zeta_1=\textbf{e}, \ \zeta_n=\zeta_{nm}^m \rangle 
\end{equation*} Ce nouveau objet amusant acquiert dans \cite{CoCaMa} une interprétation en termes du système de Bost-Connes (cf. \cite{BC}).

\subsection{$\Z$ admet une présentation finie sur $\F_1$}

On montre d'abord que $\Z$ est de type fini sur $\F_\emptyset.$ En effet, il suffit de prendre comme opérations $0\in\Sigma_\Z(\textbf{0}), [-]\in\Sigma_\Z(\textbf{1})$ et $[+]\in\Sigma_\Z(\textbf{2}),$ soumis à l'ensemble $R$ de relations exprimant que $(\Z,+)$ est un groupe abélien: 
\begin{equation*}
x+(-x)=0, \quad x+0=x=0+x, \quad (x+y)+z=x+(y+z), \quad x+y=y+x. 
\end{equation*} On a ainsi
\begin{equation}\label{pres}
\Z=\F_\emptyset\langle [0]^{[0]}, [-]^{[1]}, [+]^{[2]} \ | \ R\rangle =\F_1\langle [-]^{[1]}, [+]^{[2]} \ | \ R\rangle=\F_{1^2}\langle [+]^{[2]} \ | \ R \rangle.
\end{equation} Si l'on sous-entend que $\Z$ est une monade algébrique commutative, on n'a besoin que de deux premières relations, de sorte que
\begin{equation*}
\Z=\F_1\langle [-]^{[1]}, [+]^{[2]} \ | \ 0+\textbf{e}=\textbf{e}=\textbf{e}+0, \ (-\textbf{e})+\textbf{e}=0\rangle
\end{equation*}
Cette présentation peut être interprété comme donnant une réponse positive à la question si $\Z$ est un anneau de polynômes qui se pose en vue de l'analogie entre les corps de nombres et les corps de fonctions discuté au début. En effet, on peut voir $\Z$ comme un quotient d'un anneau de polynômes sur $\F_1,$ dont les deux variables sont des opérations, par l'\textit{idéal} engendré par les relations ci-dessus. En particulier, $\Z$ est de type fini sur $\F_1.$ C'est alors immédiat de montrer que:

\begin{proposition} Un anneau $R$ est absolument de type fini si et seulement s'il est de type fini sur $\Z$ au sens classique, c'est-à-dire, s'il existe des éléments $x_1,\cdots,x_n\in R$ tels que $R=\Z[x_1,\ldots,x_n].$ 
\end{proposition}   

De même, on obtient la présentation 
\begin{equation}\label{nat}
\N=\F_1\langle [+]^{[2]} \ | \ 0+\textbf{e}=\textbf{e}=\textbf{e}+0 \rangle.
\end{equation}

\smallskip

Le fait que $\Z$ soit de type fini sur $\F_1$ entraînera évidement que $\Spec\Z$ est de type fini sur $\Spec\F_1$, après avoir donné un sens rigoureux au point absolu $\Spec\F_1$. Cela contredit les prévision de Manin dans \cite{Ma4}, où la question sur quelle est la dimension de $\Spec\Z$ est abordée. Du point de vue plus orthodoxe, la dimension de Krull de $\Spec\Z$ est un et les nombres premiers sont des points zéro-dimensionnels dans $\Spec\Z$, que l'on peut considérer géométriquement comme les images du morphismes $\Spec\F_p \longrightarrow \Spec\Z$. Cependant, quelques analogies surprenantes entre les nombres premiers et les noeuds \cite{Mo}, ainsi que l'existence d'une dualité de Poincaré en dimension $3$ pour la topologie étale de $\Spec\Z$, suggèrent qu'une autre possible réponse est $\dim \Spec\Z=3$. Finalement, Manin a conjecturé que la dimension de $\Spec\Z$ est infini lorsque l'on regarde le spectre des entiers sur le point absolu. Dans l'approche de Durov, $\Spec\F_1$ admet une présentation de type fini, mais ce n'est pas clair comment on peut définir une notion de dimension. Remarquons, en tout cas, que $\Z$ est engendré par trois opérations.          

\smallskip

On montre l'utilité de décrire une monade par une présentation par des opérations et par des relations en calculant le produit tensoriel $\Z \otimes_{\F_1} \Z$ dans la catégorie des anneaux généralisés. Prenons d'abord la première présentation dans \ref{pres}. La monade commutative $\Z \otimes_{\F_\emptyset} \Z$ est donc engendré par deux constantes $c, d$ deux opérations unaires $u, v$ et deux opérations binaires $s$ et $t$ soumis aux mêmes relations. D'après l'exemple \ref{const}, la commutativité implique $c=d.$ On peut donc considérer $\Z \otimes_{\F_1} \Z$. En remplacent $y$ et $z$ par la constante $0$ dans l'équation de l'exemple \ref{bi} et en utilisant les relations, on en déduit $s(x,t)=t(x,t)$ pour tout $x,t$ dans un module $X,$ donc $s=t$. Un raisonnement pareil montre que $u$ et $v$ doivent aussi coïncider, de sorte que $\Z \otimes_{\F_1} \Z=\Z.$ Il s'agit d'un résultat en quelque mode décevant car l'on s'attendait de trouver une catégorie où le produit tensoriel de $\Z$ par soi même n'était pas trivial.   

\subsection{Les anneaux généralisés $A_N$ et $B_N.$}

Dans cette épigraphe on étudie deux anneaux qui vont être essentiels dans la construction de la compactification de $\Spec\Z.$ Soit $N\geq 1$ entier, définissons $f=\frac{1}{N}.$ D'abord, on pose $B_N=\Z[f],$ qui est donc un anneau au sens classique ayant 
\begin{equation*}
\Spec B_N=(\bigcup_{p \nmid N} pB_N) \cup (0) 
\end{equation*} On définit ensuite $A_N$ comme l'intersection $A_N:=B_N\cap \Z_{(\infty)}=B_N\cap \Z_\infty$. Tout élément de $B_N(\textbf{n})$ peut être écrit comme un $n$-uplet $(\frac{u_1}{N^k},\ldots,\frac{u_n}{N^k}),$ où $u_i \in \Z$ et $k\geq 0$ est le plus grand exposant qui apparaît dans les dénominateurs. Alors, l'intersection avec $\Z_\infty$ impose la contrainte $\sum_{i=1}^n \abs{u_i}\leq N^k$ et l'on en déduit:  
\begin{align}\label{grad}
A_N(\textbf{n})&=\{(\lambda_1,\ldots,\lambda_n)\in \Q^n: \lambda_i \in B_N, \sum_{i=1}^n \abs{\lambda_i}\leq 1\}= \nonumber \\
&=\{(\frac{u_1}{N^k},\ldots,\frac{u_n}{N^k}): k\geq 0, u_i\in \Z, \sum_{i=1}^n \abs{u_i}\leq N^k\}
\end{align} Notamment le monoïde sous-jacent à $A_N$ est l'ensemble $\abs{A_N}$ des éléments de $B_N$ à valeur absolue plus petite ou égale à 1. On a un plongement canonique $A_N \longrightarrow B_N,$ induisant donc un morphisme injectif des localisations (voir la section suivante pour une définition précise) $A_N[f^{-1}]\longrightarrow B_N[f^{-1}]=B_N.$ Soit $\lambda=(\lambda_1,\ldots,\lambda_n)\in B_N(\textbf{n}).$ Il existe un entier $k\geq 0$ tel que $\abs{\lambda_1}+\ldots+\abs{\lambda_n}\leq N^k.$ Alors $f^k\lambda=N^{-k}\lambda$ appartient à $A_N(\textbf{n}),$ ce qui revient à dire que $\lambda \in A_N[f^{-1}](\textbf{n}).$ Par conséquent, $B_N$ est une sous-monade de $A_N[f^{-1}],$ d'où l'égalité $B_N=A_N[f^{-1}].$      

\smallskip

\begin{proposition} L'anneau généralisé $A_N$ satisfait les propriétés suivantes:
\begin{enumerate}
\item C'est un anneau local généralisé d'idéal maximal $\mathfrak{p}_\infty=\{\lambda\in B_N: \abs{\lambda}\leq 1\}$.
\item À l'exception de $\mathfrak{p}_\infty$, les idéaux premiers non nuls de $A_N$ sont en bijection avec les nombres premiers $p$ ne divisant pas $N$.
\end{enumerate}
\end{proposition}

\begin{proof} Pour la première partie de la proposition, il suffit de considérer le plongement $j: A_N \longrightarrow \Z_{(\infty)}.$ Comme on a vu dans l'exemple 4.2, $\Z_{(\infty)}$ est un anneau local généralisé d'idéal maximal $\mathfrak{m}_{(\infty)}=\{\lambda \in \Q: \abs{\lambda}<1\}.$ L'image réciproque de $\mathfrak{m}_{(\infty)}$ par $j$ est l'idéal maximal $\mathfrak{p}_\infty:=j^{-1}(\mathfrak{m}_{(\infty)})=\{\lambda\in \abs{B_N}: \abs{\lambda}<1\}$. Comme $\abs{A_N}-\mathfrak{p}_\infty=\{-1,1\}$ est l'ensemble des éléments inversibles de $\abs{A_N},$ on conclut que $A_N$ est local. Pour la deuxième partie, soit $p$ premier ne divisant pas $N$. Alors $pB_N$ est un idéal premier de $B_N.$ Si l'on pose $\mathfrak{p}_p:=i^{-1}(pB_N)$, $i$ étant le plongement canonique de $A_N$ dans $B_N$, on obtient un idéal premier de $A_N$. Puisque $B_N$ est la localisation de $A_N$ dans le système multiplicatif engendré par $\frac{1}{N}$, les idéaux premiers de $A_N$ ne contenant pas $\frac{1}{N}$ sont en bijection avec $\Spec B_N$. La preuve se réduit donc à montrer que le seul idéal premier de $A_N$ contenant $\frac{1}{N}$ est l'idéal maximal, ce qui revient à dire que si $\mathfrak{p}$ est un tel idéal, alors $\mathfrak{p}_\infty \subset \mathfrak{p}$. Soit $\lambda \in \mathfrak{p}_\infty$. Comme $\abs{\lambda}<1$ il existe $k$ suffisamment grand tel que $\abs{\lambda}^k <\frac{1}{N}$. Alors $\mu=N\lambda^k \in \abs{A_N},$ donc $\lambda^k \in \mathfrak{p}$. Mais $\mathfrak{p}$ étant premier, cela implique $\lambda \in \mathfrak{p}$, donc $\mathfrak{p}_\infty \subset \mathfrak{p}.$
\end{proof}

\smallskip
La conséquence la plus importante de la proposition ci-dessus est la description du spectre de l'anneau généralisé $A_N$, où l'on voit apparaître pour la première fois une réalisation concrète du premier infini: 

\begin{equation*}
\Spec A_N=\{(0)\}\cup \{\mathfrak{p}_p: p\nmid N\}\cup \{\mathfrak{p}_\infty\}
\end{equation*} Dorénavant, on note $\xi:=(0)$, $p:=\mathfrak{p}_p$ et $\infty:=\mathfrak{p}_\infty$ par analogie avec une courbe projective. Alors $\Spec A_N$ contient un point générique $\xi$ et un seul point fermé $\infty$, dont le complément est homéomorphe à $\Spec B_N=\Spec \Z[\frac{1}{N}]$. Par conséquent, $\overline{p}=\{p,\infty\}$.  

\smallskip

\begin{corollaire} Soit $U \subsetneq \Spec A_N$ une partie non-vide. Alors $U$ est ouverte si et seulement si $\Spec A_N-U$ contient $\infty$ et un nombre fini d'idéaux $p$.  
\end{corollaire}

On finit l'étude de l'anneau généralisé $A_N$ en énonçant un théorème, dont la démonstration se trouve dans \cite[7.1.26]{Du}, qui sera important dans la discussion sur la \textit{dimension} de la compactification de $\Spec\Z$. 

\begin{theoreme}\label{an} L'anneau généralisé $A_N$ admet une présentation finie sur $\F_1,$ engendrée par les opérations 
\begin{equation*}
s_p(\{1\},\ldots,\{p\})=\frac{1}{p}\{1\}+\ldots+\frac{1}{p}\{p\},
\end{equation*} où $p$ est un diviseur premier de $N,$ et les relations 
\begin{align*}
s_p(\{1\},\ldots,\{1\})&=\{1\} \\
s_p(\{1\},\ldots,\{p\})&=s_n(\{\sigma(1)\},\ldots,\{\sigma(p)\}), \ \quad \sigma \in \mathcal{S}_p \\
s_p(\{1\},\ldots,\{p-1\},-\{p-1\})&=s_n(\{1\},\ldots,\{p-2\},0,0)
\end{align*}
\end{theoreme}
  
\begin{table}[ht]
\caption{Exemples d'anneaux généralisés}
\centering
\begin{tabular}{c c c c c}
\\[0.5ex] \hline 
Monade & Hypoadditive & Hyperadditive & Corps & Présentation \\
\hline 
& & & & \\
$\Z$ & oui & oui & non & (\ref{pres}) \\[0.3ex]
$\N$ & oui & oui & non & (\ref{nat}) \\[0.3ex]
$\Z_\infty$ & oui & non & non & ? \\[0.3ex]
$\Z_{(\infty)}$ & oui & non & non & $\F_{1^2}\langle \{s_n\}_{n>1} \rangle$ \\[0.3ex]
$\overline{\Z}_\infty$ & oui & non & non & ? \\[0.3ex]
$\F_\emptyset$ & non & non & non & monade initiale \\[0.3ex]
$\F_1$ & oui & non & oui & $\F_\emptyset[0^{[0]}]$\\[0.3ex]
$\F_{1^n}$ & oui & non & non & $\F_1\langle \zeta_n^{[1]} | \zeta_n^nx=x \rangle$ \\[0.3ex]
$\F_{1^{\infty}}$ & oui & non & non & non finie \\[0.3ex]
$\F_\infty$ & non & non & oui & (\ref{res}) \\[0.3ex]
$A_N$ & oui & non & non & theor. \ref{an} \\[0.3ex]
$B_N$ & oui & oui & non & $B_N=\Z[N^{-1}]$ \\
& & & & \\
\hline 
\end{tabular}
\end{table}

\section{Vers la compactification de $\Spec\Z$}

\subsection{Schémas généralisés}

On a déjà remarqué dans l'introduction que l'étape délicate de la nouvelle théorie consistait en définir une catégorie juste pour remplacer les anneaux commutatifs usuels. Après cela, la construction d'un schéma affine généralisé est totalement analogue à celle provenant de la géométrie algébrique classique. Le travail le plus dur a été, donc, déjà fait. Dans cette section, on définit d'abord ce que l'on entend par localisation et par idéaux premiers, puis on construit le spectre d'un anneau généralisé comme espace topologique et on le munit d'un faisceau d'anneaux généralisés, ce qui nous permet d'arriver à la définition de espace annelé généralisé.    

\smallskip

Soit $A$ un anneau généralisé et soit $S \subset \abs{A}$ un sous-monoïde. On appelle localisation de $A$ l'objet initial $A[S^{-1}]$ dans la catégorie des couples $(B,\rho),$ où $B$ est un anneau généralisé et $\rho: A \longrightarrow B$ est un morphisme tel que tous les éléments de $\rho_1(S)\subset \abs{B}$ sont inversibles. Lorsque $S$ est le système multiplicatif engendré par un élément $f\in\abs{A},$ on écrit $A[f^{-1}]$ ou $A_f$. Une description plus explicite est la donnée par la présentation suivante, dans laquelle on considère les opérations unaires $s^{-1}$ pour tout $s\in S:$ 
\begin{equation*}
A[S^{-1}]=A\langle (s^{-1})_{s\in S} | ss^{-1}=\textbf{e}=s^{-1}s \rangle. 
\end{equation*} Comme dans le cas classique, $A[S^{-1}](\textbf{n})$ contient les classes d'équivalence des couples $(a,s)\in A(\textbf{n}) \times S$ modulo la relation $(x,s)\sim (y,t) \Leftrightarrow \exists u\in S$ tel que $utx=usy$. 

\smallskip

Soit $A$ un anneau généralisé. On appelle idéal de $A$ tout $A$-sous-module du monoïde sous-jacent $\abs{A}$. En imitant la théorie usuelle, les idéaux premiers sont les idéaux $\mathfrak{p}$ tels que $S_\mathfrak{p}:=\abs{A}-\mathfrak{p}$ est un système multiplicatif, et les idéaux maximaux sont les éléments maximaux pour l'ordre partiel défini par l'inclusion dans l'ensemble d'idéaux propres de $A$. Comme dans l'algèbre commutative, tout idéal maximal est premier et tout anneau généralisé admet au moins un idéal maximal. Quand il n'y en a plus, l'anneau est dit local. Cela équivaut à dire que tous les éléments du complément de l'idéal dans $\abs{A}$ sont inversibles. On est alors en état de construire le spectre premier d'un anneau généralisé:

\smallskip

\begin{definition} Soit $A$ un anneau généralisé. On appelle \textit{spectre premier} de $A$, et on le note $\Spec A,$ l'ensemble d'idéaux premiers de $A$ muni de la topologie, dite de Zariski, dont les fermés sont les ensembles
\begin{equation*}
V(M):=\{\mathfrak{p}\in \Spec A: M \subset \mathfrak{p}\},
\end{equation*} $M$ étant un sous-ensemble quelconque de $\abs{A}$. Une base de cette topologie est formée des ouverts principaux $D(f):=\{\mathfrak{p}: f\notin \mathfrak{p}\},$ où $f\in\abs{A}$.  
\end{definition}  

\smallskip

Chaque $A-$module $M$ définit un faisceau d'anneaux généralisés $\tilde{M}$ sur $\Spec A$ dont les sections sont $\Gamma(D(f),M)=M_f$ pour chaque ouvert principal $D(f)$ correspondant à un élément $f$ du monoïde $\abs{A}.$ Lorsque $M=A$ on obtient le faisceau structural $\mathcal{O}_{\mathrm{Spec}A}.$ Tous les $\tilde{M}$ sont alors des $\mathcal{O}_{\mathrm{Spec} A}$-modules.

\smallskip

\begin{definition} Un \textit{espace annelé généralisé} est la donnée d'un couple $(X, \mathcal{O}_X)$ formé d'un espace topologique et d'un faisceau d'anneaux généralisés sur $X$. Un morphisme d'espaces annelés généralisés est un couple $(f,f^\sharp): (X, \mathcal{O}_X)\longrightarrow (Y,\mathcal{O}_Y)$ formé d'une application continue $f: X \longrightarrow Y$ et d'un morphisme de faisceaux $f^\sharp: \mathcal{O}_Y \longrightarrow f_\ast\mathcal{O}_X$, où $f_\ast\mathcal{O}_X$ est l'image directe du faisceau $\mathcal{O}_X$ par $f$. On dit que $(X,\mathcal{O}_X)$ est localement annélé si pour tout $P\in X$ l'anneau des germes $\mathcal{O}_{X,P}$ est local.  
\end{definition}

\smallskip    

On obtient de même une notion d'isomorphisme d'espaces annelés généralisés. Les schémas affines généralisés sont alors les espaces annelés les plus simples et l'on va construire les schémas généralisés en les recollant:

\smallskip

\begin{definition}
Un \textit{schéma affine généralisé} $X$ est un espace localement annelé généralisé qui est isomorphe à $(\Spec A, \mathcal{O}_{\mathrm{Spec}A})$ pour un certain anneau généralisé $A$. Un $\textit{schéma généralisé}$ est alors un espace annélé généralisé $(X,\mathcal{O}_X)$ qui admet un recouvrement ouvert par des schémas affines généralisés. Lorsque $\Gamma(X,\mathcal{O}_X)$ admet une constante, ce qui revient à dire que $\Gamma(X,\mathcal{O}_X(\textbf{0}))$ est non-vide, on dit que $X$ est un schéma sur $\F_1$. 
\end{definition} 

\smallskip

Un morphisme entre deux schémas généralisés $X$ et $Y$ est un morphisme $f$ des espaces annélés généralisés $(X,\mathcal{O}_X)$ et $(Y,\mathcal{O}_Y)$ tel que pour tout ouvert affine $\Spec B=V \subset Y$ et $\Spec A=U \subset X$ vérifiant $f(V)\subset U$, la restriction $f|V: V \longrightarrow U$ soit induite par un morphisme d'anneaux généralisés $\varphi: A \longrightarrow B$. On obtient ainsi la catégorie des schémas généralisés. Comme dans la géométrie algébrique usuelle du point de vue fonctoriel, un schéma généralisé peut être défini d'une manière plus intrinsèque comme un foncteur contravariant $\mathcal{S}: \mathfrak{Gen} \longrightarrow \mathfrak{Ens}$ de la catégorie des anneaux généralisés vers la catégorie des ensembles qui est localement isomorphe au spectre d'un anneau pour la topologie de Zariski. On se refère à l'annexe A pour la construction fonctorielle complète, inspiré de \cite{ToVa}, d'un schéma généralisé sur $\F_1$. 

\subsection{Compactification à la main}

Dans les exemples précédents, on a obtenu des descriptions complètes des spectres
\begin{equation*}
\Spec \Z=\{\xi,p,\ldots\}, \quad  \Spec A_N=\{\xi, p, \ldots, \infty\}_{p\nmid N}, \quad \Spec B_N=\{\xi,p,\ldots\}_{p\nmid N}.
\end{equation*} On définit $\widehat{\Spec \Z}^{(N)}$ comme le schéma affine généralisé obtenu en recollant $\Spec \Z$ et $\Spec A_N$ le long de leurs sous-ensembles ouverts principaux homéomorphes à $\Spec B_N$: 
\begin{equation*}
\widehat{\Spec \Z}^{(N)}:=\Spec \Z \coprod_{\Spec B_N} \Spec A_N 
\end{equation*} Ainsi, $\widehat{\Spec \Z}^{(N)}=\Spec \Z \cup \{\infty\}$ contient déjà un point additionnel $\infty$ correspondant à la valuation archimédienne de $\Q$. Le point $\xi$ est de nouveau générique et les points fermés sont $\infty$ et les premiers $p$ divisant $N$. Par conséquent:

\smallskip

\begin{lemme} Une partie non-vide $U \subsetneq \widehat{\Spec\Z}^{(N)}$ est ouverte si et seulement si elle contient $\xi$, son complément est fini et soit $\infty \notin U$ soit $\Spec B_N \subset U.$ 
\end{lemme}

\begin{proof} C'est clair que le point générique est inclus dans toutes les parties ouvertes non vides et que le complément de $U$ doit être fini, puisque $\Spec \Z$ est ouvert dans $\widehat{\Spec\Z}^{(N)}$. Soit alors $U$ ouvert contenant $\xi$ à complément fini $\{p_1,\ldots,p_n\}$. Supposons qu'un des $p_i$ ne divise pas $N$. Alors, $\overline{\{p_i\}}=\{p_i,\infty\}$, donc $\infty \notin U$. Réciproquement, si $U$ est une partie de $\widehat{\Spec\Z}^{(N)}$ vérifiant les trois conditions de l'énoncé, son complément est un ensemble fini $S=\{p_1,\ldots,p_n\}$ tel que $p_i\neq \xi$ pour tout $1\leq i\leq n$ et que, soit $\infty$ est un des $p_i$ et alors $S$ est fermé, soit tous les premiers $p_i$ sont parmi les diviseurs de $N$, donc $S$ est une réunion de points fermés.    
\end{proof}

D'après le lemme, la topologie de $\widehat{\Spec \Z}$ dépend fortement du choix de $N$. Puisque $\xi$ et $p$ appartiennent au sous-schéma ouvert $\Spec \Z$, lorsque l'on calcule les anneaux des germes $\mathcal{O}_p$, on obtient le résultat attendu $\mathcal{O}_\xi=\Q$ et $\mathcal{O}_p=\Z_{(p)}$. Cependant, $\Spec A_N$ étant le voisinage ouvert à l'infini, l'anneau des germes n'est pas $\Z_{(\infty)}$ comme l'on voudrait, mais $\mathcal{O}_\infty=A_{N,\mathfrak{p}}=A_N$. Ces deux observations motivent l'idée de faire disparaître $N$ en prenant une certaine limite projective$^\star$. Ensuite, on décrit le système filtré que l'on va considérer.

\smallskip

Soient $N,M>1$ des entiers. On construit un morphisme 
\begin{equation*}
f:=f_N^{NM}: \widehat{\Spec\Z}^{(NM)} \longrightarrow \widehat{\Spec\Z}^{(N)} 
\end{equation*} en choisissant d'abord des sous-schémas ouverts $U_1\approx \Spec\Z,\ U_2\approx \Spec A_N$ de $\widehat{\Spec\Z}^{(N)}$ et $U_1'\approx \Spec \Z,\ U_2'\approx \Spec A_{NM}$ de $\widehat{\Spec\Z}^{(NM)}$ respectivement. Ainsi, 
\begin{equation*}
\widehat{\Spec\Z}^{(N)}=U_1 \cup U_2, \quad \widehat{\Spec\Z}^{(NM)}=U_1'\cup U_2', 
\end{equation*} l'intersection étant $W=U_1\cap U_2=U_1'\cap U_2'=\Spec B_N$. On pose alors $V_1=U_1'$ et $V_2=U_2'\cup W$. On a de nouveau $V_1 \cap V_2=W$. Se donner le morphisme f est donc équivalent à se donner des morphismes $f_i: V_i \longrightarrow U_i, \ i=1,2$, qui coïncident sur $W=V_1\cap V_2$ et tels que $f_i^{-1}(U_1\cap U_2)=V_1\cap V_2$.  

\smallskip

Puisque $U_1=V_1=\Spec\Z$, on peut choisir $f_1$ comme étant l'identité sur $\Spec\Z.$ Alors, $f_{1|W}=\id_W$. Pour $f_2$ on peut faire à peu près la même chose, quitte à utiliser la correspondance bijective entre les morphismes $f:V_2 \longrightarrow U_2$ dans la catégorie des schémas généralisés et les morphismes $\varphi: \Gamma(U_2,\mathcal{O}) \longrightarrow \Gamma(V_2,\mathcal{O})$ dans la catégorie des anneaux généralisés. En effet, $U_2$ étant $\Spec A_N$, le premier terme est juste $\Gamma(U_2,\mathcal{O})=A_N$. En ce qui concerne le deuxième:
\begin{equation*}
\Gamma(V_2,\mathcal{O})=\Gamma(U_2',\mathcal{O})\times_{\Gamma(U_2'\cap W,\ \mathcal{O}_{\mathrm{Spec}\Z})} \Gamma(W, \mathcal{O}). 
\end{equation*} Puisque $U_2'=\Spec A_{NM}$ et $W=\Spec B_N,$ il suffit de calculer les sections de l'intersection. On déduit de l'égalité
\begin{equation*}
\Spec A_{NM} \cap \Spec \Z \cap \Spec B_N=\Spec B_{NM} \cap \Spec B_N=\Spec B_N 
\end{equation*} que $\Gamma(U_2'\cap W,\mathcal{O}_{\mathrm{Spec}Z})=B_{NM}$. Par conséquent:
\begin{equation*}
\Gamma(V_2,\mathcal{O})=A_{NM}\times_{B_{NM}} B_N=A_{NM}\times_\Q B_N=A_{NM} \cap B_N=A_N,
\end{equation*} où l'on a appliqué encore une fois le fait que $\Gamma(\Spec \mathrm{R}, \mathcal{O})=\mathrm{R}$ (cf. \cite[Prop. 2.2]{Har}), qui reste encore valable dans le cas des anneaux généralisés. Cela montre que les morphismes $f_2: V_2 \longrightarrow U_2$ sont en bijection avec les morphismes $\varphi: A_N \longrightarrow B_N$. On choisit donc l'image $f_2$ de $\varphi=\id_{A_N}$ par cette correspondance. La même construction prouve que $f_{2|W}=\id_W$. En effet, $f_{2|W}$ est induite par le plongement canonique de $A_N$ dans $B_N$, donc il s'agit d'une immersion ouverte de $\Spec B_N$ dans $\Spec A_N$.    

\smallskip

Il reste à vérifier la condition sur les images réciproques. Dans le premier cas, on a 
\begin{equation*}
f_1^{-1}(U_1\cap U_2)=f_1^{-1}(D_{U_1}(N))=D_{V_1}(N)=W=V_1\cap V_2.
\end{equation*} Dans le deuxième cas, 
\begin{align*}
f_2^{-1}(U_1 \cap U_2)&=f_2^{-1}(D_{U_2}(1/N))=D_{U_2'}(1/N)\cup D_W(1/N)=\Spec A_{NM}[(1/N)^{-1}]\cup W \\
&=\Spec B_{NM} \cup \Spec B_N=\Spec B_N=W,
\end{align*} car d'après 4.7, $B_{NM}$ est la localisation de $A_{NM}$ dans $\frac{1}{N}$. Cela complète la construction du morphisme $f$. Lorsque $M|N^k$ pour quelque $k\geq 1$, les anneaux généralisés $A_{NM}$ et $B_{NM}$ coïncident avec $A_N$ et $B_N$, donc $f_N^{NM}$ est l'identité. Autrement, il existe un premier $p$ divisant $M$ mais pas $N$. Ce point est fermé dans $\widehat{\Spec\Z}^{(NM)}$ mais pas dans $\widehat{\Spec\Z}^{(N)}$, ce qui entraîne que l'application $f_N^{NM}$ n'est pas un homéomorphisme, donc elle n'est pas un isomorphisme non plus. 

\smallskip

Ordonnons l'ensemble des entiers plus grands que 1 par la relation de divisibilité. Si $M$ divise $N$ et l'on note par $K$ le quotient, soit $f_M^N:=f_M^{MK}$, avec la convention $f_N^N=\id$ lorsque $N=M$. Ces fonctions étant transitives, on obtient le système projectif filtré
\begin{equation*}
\widehat{\Spec\Z}^\bullet=(\{\widehat{\Spec\Z}^{(N)}\}_{N>1}, \{f_M^N\}_{M|N})
\end{equation*} 

\smallskip

\begin{definition}
On appelle \textit{compactification du spectre} de $\Z$ la limite projective du système ci-dessus dans la catégorie des pro-schémas généralisés, c'est-à-dire: 
\begin{equation*}
\widehat{\Spec \Z}:=\varprojlim_{N>1} \widehat{\Spec \Z}^{(N)}
\end{equation*}
\end{definition}     

\smallskip

On étudie d'abord la topologie de $\widehat{\Spec\Z}$. Puisque tous les objets du système ont le même ensemble sous-jacent, au niveau des ensembles, la compactification de $\Spec\Z$ n'est autre chose que $\Spec\Z\cup \{\infty\}$. Dans $\widehat{\Spec\Z}^{(N)}$ un point $p$ était fermé si et seulement si $p|N$. En traitant tous les entiers plus grand que 1 en même temps dans la construction de la limite, on peut toujours choisir un multiple $M$ de $p$ de sorte que le point soit fermé dans tous les $\widehat{\Spec\Z}^{N}$ tels que $M|N$. Cela démontre le lemme suivant:

\begin{lemme} Tous les points de $\widehat{\Spec\Z}$ sauf le point générique $\xi$ sont fermés. Par conséquent, une partie non-vide $U$ de $\widehat{\Spec\Z}$ est ouverte si et seulement si elle contient $\xi$ et son complément est fini. 
\end{lemme} 

En ce sens, $\widehat{\Spec \Z}$ ressemble beaucoup à une courbe projective sur $\Z$, ce que l'on voulait avoir pour renforcer l'analogie entre les corps des nombres et les corps de fonctions. De plus, lorsque l'on calcule l'anneau des germes à l'infini, on trouve la limite
\begin{equation*}
\mathcal{O}_{\widehat{\Spec\Z},\infty}=\varinjlim_{N>1} \mathcal{O}_{\widehat{\Spec\Z}^{(N)},\infty}=\varinjlim A_N=\Z_{(\infty)},
\end{equation*} car l'on a défini $A_N$ comme étant l'intersection $\Z_{(\infty)}\cap B_N$ et la limite de $B_N=\Z[1/N]$ lorsque $N\rightarrow \infty$ est bien $\Z$.   
   
\subsection{Description en termes de la construction $\mathrm{Proj}$}

Le fait que la topologie de $\widehat{\Spec\Z}$ soit la même que celle d'une courbe projective suggère la possibilité d'imiter la construction $\mathrm{Proj}$ des schémas projectives afin d'avoir une description plus intrinsèque de la compactification de $\Spec\Z$. Soient $\mathcal{C}$ une catégorie monoïdale et $\Delta$ un monoïde commutatif tels que les coproduits indexés par des sous-ensembles de $\Delta$ existent dans $\mathcal{C}$ et qu'ils commutent avec $\otimes$. Une algèbre $\Delta$-graduée dans $\mathcal{C}$ est alors une famille $\{S^\alpha\}_{\alpha\in\Delta}$ d'objets de $\mathcal{C}$ munis d'un morphisme d'identité $\epsilon: \id_\mathcal{C} \longrightarrow S^0$ et de morphismes de multiplication $\mu_{\alpha,\beta}: S^\alpha \otimes S^\beta \longrightarrow S^{\alpha+\beta}$ vérifiant les relations d'associativité et d'unité évidentes. Lorsque l'on considère la catégorie monoïdale des endofoncteurs sur les ensembles, on obtient des monades graduées:

\begin{definition} Soit $\Delta$ un monoïde commutatif, par exemple $\N$ ou $\Z$. Une collection d'ensembles $\{R^\alpha(\textbf{n})\}_{\alpha\in\Delta, n\geq 0}$ et d'applications $S^\alpha(\varphi): S^\alpha(\textbf{n}) \longrightarrow S^\alpha(\textbf{m})$ définies pour chaque $\varphi: \textbf{n} \longrightarrow \textbf{m}$ tels que le foncteur 
\begin{equation*}
S^\alpha: \underline{\N} \longrightarrow \Ens, \quad \textbf{n} \longmapsto S^\alpha(\textbf{n})
\end{equation*} admet une extension unique à un endofoncteur algébrique $S^\alpha$ est une monade algébrique $R$ s'il existe un morphisme d'identité $\epsilon: \id_\Ens \longrightarrow S^0$ et des morphismes de multiplication $\mu_{n,\beta}^{(k,\alpha)}: S^\alpha(\textbf{k})\times S^\beta(\textbf{n})^k \longrightarrow S^{\alpha+\beta}(\textbf{n})$ vérifiant les versions graduées des axiomes de monade algébrique. Le monoïde sous-jacent à $R$ est le monoïde gradué $\abs{R}=\sqcup_{\alpha \in \Delta} \abs{R^\alpha}$ et l'ensemble d'opérations est l'ensemble gradué $||R||=\sqcup_{\alpha \in \Delta} ||R^\alpha||$. On dit que $R$ est commutative si toutes les opérations de $||R||$ commutent au sens de la définition 10. On obtient alors des anneaux généralisés gradués.  
\end{definition} 

Soit $R$ un anneau généralisé $\N$-gradué. Notons par $\abs{R}^+$ le sous-monoïde de $\abs{R}$ formé des éléments de degré non nul. Pour chaque $f\in \abs{R}^+$, soit $R_{(f)}$ la composante de degré zéro de la localisation $R_f=R[f^{-1}]$. Remarquons que si $g\in \abs{R_{(f)}}$, alors $(R_{(f)})_{(g)}$ est isomorphe à $R_{(fg)}$. Le monoïde sous-jacent à $R$ étant commutatif, on obtient le même résultat lorsque l'on localise d'abord en $g$ et puis en $f$. Soit $D_+(f)$ le schéma affine généralisé $\Spec R_{(f)}$. Si $g \in \abs{R}^+$ est un autre élément, $D_+(f)$ et $D_+(g)$ contiennent des sous-schémas ouverts isomorphes à $D_+(fg)$, donc on peut les recoller le long $D_+(fg)$.

\begin{definition} Soit $R$ un anneau généralisé gradué. On appelle \textit{spectre projectif} de $R$   
\begin{equation*}
\mathrm{Proj\ } R=\coprod_{D_+(fg), \ f,g \in \abs{R}^+} D_+(f),
\end{equation*} qui est un schéma généralisé ayant un recouvrement ouvert affine $\{D_+(f)\}_{f\in \abs{R}^+}$ tel que $D_+(f)\cap D_+(g)\cong D_+(fg)$. 
\end{definition}

\begin{exemple} La construction ci-dessus permet de définir l'espace projectif sur un anneau généralisé $R$ quelconque comme le spectre projectif de l'algèbre de polynômes à $n+1$ indéterminées de degré un, autrement dit: 
\begin{equation*}
\mathbb{P}_R^n:=\mathrm{Proj\ } R[T_0^{[1]},\ldots,T_n^{[1]}].
\end{equation*} D'après la définition, la famille $\{D_+(T_i)\}_{0\leq i \leq n}$ constitue un recouvrement ouvert de $\mathbb{P}_R^n$. Ce n'est pas difficile à voir que chaque $D_+(T_i)$ est isomorphe à $\mathbb{A}^n_R:=\Spec[T_1^{[1]},\ldots,T_n^{[1]}]$. Notamment, la droite projective $\mathbb{P}_{\F_1}^1=\mathrm{Proj\ } \F_1[T_0^{[1]},T_1^{[1]}]$ est obtenu en recollant deux copies de $\A^1_{\F_1}=\Spec \F_1[T]=\{(0), (T)\}$ le long du point générique. Elle contient donc trois points $\{\xi, 1,\infty \}$ correspondant aux idéaux $(T_0)$, $(T_1)$ et $(0)$, dont les deux premiers sont fermés. En général, le cardinal de l'espace projectif $\mathbb{P}^n(\F_1)$ construit ainsi est $2^{n+1}-1$, ce qui n'accorde pas avec le point de vue de Tits.     
\end{exemple}

Après cet interlude théorique, on est en état de décrire $\widehat{\Spec\Z}$ comme un (pro)schéma projectif. En vue de la seconde égalité dans \ref{grad}, définissons $R$ comme l'anneau généralisé gradué dont la composante en degré $d$ est donné par 
\begin{equation*}
R_d(\textbf{n})=\{(\lambda_1,\ldots,\lambda_n)\in \Z^n: \sum_{i=1}^n \abs{\lambda_i}\leq N^d\} 
\end{equation*} On peut plonger $R$ dans $\Z[T]$ en identifiant $(\lambda_1,\ldots,\lambda_n)$ avec $T^d(\lambda_1,\ldots,\lambda_n)$. Les éléments de degré un de $R$ sont alors ceux de la forme $uT^d$, où $d\geq 0, u\in\Z, \abs{u}\leq N^d$. En particulier, $f_1:=T$ et $f_2:=NT$ appartient à $\abs{R}$. Alors, on peut montrer \cite[7.1.44]{Du}: 

\begin{lemme} $\abs{R}^+$ coïncide avec le radical de l'idéal engendré par $f_1$ et $f_2$. 
\end{lemme}  

Par conséquent, il suffit de calculer les localisations $R_{(f_1)}, R_{(f_2)}$ et $R_{(f_1f_2)}$. Commençons par $R_{(f_1)}(\textbf{n})$, qui est la réunion
\begin{equation*}
\bigcup_{d\geq 0} \{ T^d(\lambda_1,\ldots,\lambda_n)/T^d: \sum_{i=1}^n \abs{\lambda_i}\leq N^d \}
\end{equation*} L'inclusion $R_{(f_1)}(\textbf{n})\subset \Z^n=\Z(\textbf{n})$ étant évidente, il suffit de montrer la réciproque. Or, si l'on se donne un $n$-uplet $(\lambda_1,\ldots,\lambda_n)\in \Z^n$, il existe $d\geq 0$ tel que $\sum_{i=1}^n \abs{\lambda_i}\leq N^d$, donc $\Z(\textbf{n})\subset R_{(f_1)}$. On en déduit que $R_{(f_1)}=\Z$. Des raisonnements tout à fait analogues montrent que $R_{(f_2)}=A_N$ et que $R_{(f_1f_2)}=B_N$. Alors, $\mathrm{Proj\ } R$ est le recollement des ouverts $D_+(f_1)=\Spec\Z$ et $D_+(f_2)=\Spec A_N$ le long de $D_+(f_1f_2)=\Spec B_N$, de sorte que 
\begin{equation*}
\mathrm{Proj\ }R=D_+(f_1) \coprod_{D_+(f_1f_2)} D_+(f_2)=\Spec \Z \coprod_{\Spec B_N} \Spec A_N
\end{equation*} est isomorphe au schéma généralisé $\widehat{\Spec\Z}^{(N)}$. Cela démontre le théorème suivant:

\begin{theoreme} La compactification de $\Spec\Z$ est un pro-schéma projectif sur $\F_1$. 
\end{theoreme}

\subsection{Compactification d'Arakelov des variétés et fibrés en droite}

Dans cette section, on commence par revisiter l'analogie entre les corps de nombres et les corps de fonctions. Si $X$ est une variété algébrique lisse et projective sur un corps de fonctions $K=k(C)$, on appelle \textit{modèle} de $X$ un schéma projectif plat $\mathscr{X} \longrightarrow C$ dont la fibre générique $\mathscr{X}_\xi$ est isomorphe à $X$. On a déjà remarque que pour un certain nombre d'arguments (par exemple, si l'on veut utiliser la théorie de l'intersection afin d'estimer le nombres de points rationnels de $X$), il est indispensable que la courbe $C$ soit propre. Puisque tout corps de fonctions est une extension finie de $k(T)$, le problème consiste essentiellement en construire un modèle propre sur $K=k(T)$, puis on peut considérer la normalisation dans des corps plus grands. En procédant par analogie, le premier pas consiste en définir un modèle propre de $\Q$. Le premier candidat était $\Spec \Z$, dont le corps de fonctions rationnelles est $\Q$, mais on a montré qu'il n'est pas propre. Cela est à l'origine de la construction de $\widehat{\Spec\Z}$, que l'on peut voir comme un modèle propre et lisse de $\Q$. La première étape étant accomplie, il s'agit maintenant de trouver des modèles d'une variété algébrique définie sur $\Q$, que l'on peut appeler sa compactification d'Arakelov. Remarquons que dans le cadre des corps de nombres, on ne peut pas appliquer des procédés géométriques comme la normalisation.       

\smallskip

Soit $X$ une variété algébrique définie sur $\Q$. Puisque la catégorie des schémas $\mathscr{X}$ sur $\widehat{\Spec\Z}$ admettant une présentation finie est équivalente à la catégorie des triplets $(\mathscr{X}, \mathscr{X}_\infty, \theta)$, où $\mathscr{X}$ et $\mathscr{X}_\infty$ sont des schémas ayant une présentation finie sur $\Spec\Z$ et $\Spec\Z_{(\infty)}$ respectivement et $\theta: \mathscr{X}(\Q) \longrightarrow \mathscr{X}_\infty(\Q)$ est un isomorphisme de $\Q$-schémas \cite[7.1.23]{Du}, trouver un modèle de $X$ sur la compactification de $\Spec\Z$ est en fait la même chose que trouver des modèles de $X$ sur $\Spec\Z$ et sur $\Spec\Z_{(\infty)}$. Étant donné que la géométrie algébrique classique s'est occupé avec succès de la première partie, on se limite ici à décrire les modèles sur $\Spec\Z_{(\infty)}$. 

\begin{theoreme} Toute variété algébrique affine ou projective sur $\Q$ admet au moins un modèle ayant présentation finie sur $\widehat{\Spec\Z}$. 
\end{theoreme}

\begin{proof} D'après la remarque précédente, il suffit de construire un modèle sur $\Spec \Z_{(\infty)}$, qui est le spectre de l'anneau de valuation archimédien sur $\Q$. On rappelle que l'anneau de polynômes $\Q[T_1,\ldots,T_k]$ admet une norme
\begin{equation*}
||P||=\sum_{\alpha=(\alpha_1,\ldots,\alpha_k)} \abs{c_\alpha}, \quad \textit{où} \quad P=\sum_{\alpha=(\alpha_1,\ldots,\alpha_k)}c_\alpha T^\alpha.
\end{equation*} Donc, on peut identifier l'anneau généralisé de polynômes $\Z_{(\infty)}[T_1,\ldots,T_k]$ avec les polynômes dans $\Q[T_1,\ldots,T_n]$ ayant norme plus petit ou égale à 1.   
 
\begin{enumerate}
\item (Cas affine). Soit $X=\Spec A$, où $A=\Q[T_1,\ldots,T_k]/(f_0,\ldots,f_m)$, une variété algébrique affine sur $\Q$. Quitte à multiplier par des scalaires non nuls dans $\Q$, on peut supposer que $||f_j||\leq 1$, ce qui revient à dire que $f_j \in \Z_{(\infty)}[T_0,\ldots,T_k]$. Soit
\begin{equation}\label{mod}
B=\Z_{(\infty)}[T_0,\ldots,T_k| f_1=0,\ldots,f_m=0]
\end{equation} Si l'on pose $\mathscr{X}=\Spec B$, on obtient un schéma affine généralisé ayant une présentation finie sur $\Z_{(\infty)}$ dont la fibre générique est 
\begin{equation*}
\mathscr{X}_{(K)}=\Spec(B\otimes_{\Z_{(\infty)} \Q})=\Spec A=X.
\end{equation*} Par conséquent, $\mathscr{X}/\Z_{(\infty)}$ est un modèle de $X/\Q$. 

\item (Cas projectif). Soit maintenant $X/\Q$ une variété projective. Alors $X=\mathrm{Proj\ }A$, où $A$ est le quotient de $\Q[T_0,\ldots,T_k]$ par un idéal homogène $(f_1,\ldots,f_m)$, et l'on peut supposer de nouveau que $||f_j||\leq 1$. Dans ce cas, $f_j \in \Z_{(\infty)}[T_0,\ldots,T_k]$. Définissons $B$ comme dans (\ref{mod}). Alors, le spectre projectif de l'anneau généralisé gradué $B$ est un modèle $\mathscr{X}=\mathrm{Proj\ } B$ de $X$ sur $\Z_{(\infty)}$. \qedhere
\end{enumerate} 
\end{proof}

\begin{remarque} Comme le lecteur aura déjà remarqué, la preuve du théorème ci-dessus n'utilise aucune propriété spéciale de $\Q$. En fait, on peut étendre le résultat à un corps $K$ quelconque et à son anneau de valuation archimédien $\mathcal{O}_\infty$ au sens de l'exemple 4.3. 
\end{remarque}

\smallskip

Puisque dans le cadre géométrique des corps de fonctions les fibrés en droite sur une courbe projective lisse se sont révélés un outil indispensable, on voudrait avoir des analogues arithmétiques. D'après la théorie d'Arakelov classique, un fibré en droite sur l'hypothétique compactification de $\Spec\Z$ est un fibré en droite sur $\Spec \Z$ muni de quelques données archimédiennes supplémentaires. Plus précisément, un fibré en droite sur $\Spec\Z$ est un couple $L=(L,<\cdot>_E)$ formé d'un $\Z$-module libre de rang fini $L$ et d'un produit scalaire hermitien sur $E\otimes_\Z \C$ qui est invariant sous la conjugaison complexe, donc qui définit un produit scalaire euclidien sur $E\otimes_\Z \R$. En posant comme d'habitude $||x||=<x,x>_E$, on peut voir un fibré en droite sur la compactification de $\Spec\Z$ comme un réseau euclidien $(\Z^d,||\cdot||)$. Cela suggère du coup une interprétation de l'information additionnelle en termes des $\Z_\infty$-structures introduites dans 4.2. 

\begin{definition} Un fibré en droite sur $\widehat{\Spec\Z}$ est un couple $L=(L,A)$ formé d'un fibré en droite sur $\Spec\Z$, i.e. d'un $\Z$-module libre de rang fini, et d'un $\Z_\infty$-réseau $A$ (i.e. d'un corps compact symétrique convexe) dans le $\R$-espace vectoriel $L\otimes_\Z \R$.
\end{definition}         

Alors, les sections globales correspondent à l'intersection $L\cap A$. Le comptage du nombre de points d'un réseau appartenant à un corps compact symétrique est un sujet d'étude classique. On a, par exemple, le théorème de Minkowski, d'après lequel un corps convexe symétrique dans $\R^n$ de volume plus grand que $2^n$ contient au moins un point du réseau $\Z^n$. Ce résultat, qui implique que chaque classe du groupe de classes d'idéaux d'un corps de nombres $K$ admet un représentant intégral de norme borné par une constante qui ne dépend que du discriminant et de la signature de $K$, peut être imaginé comme un théorème de Riemann-Roch pour les fibrés en droite sur $\widehat{\Spec\Z}$. On rencontre de nouveau le leitmotiv \textit{la géométrie devient combinatoire sur le corps à un élément}.     

\smallskip

La définition de fibré en droite sur la compactification de $\Spec\Z$ admet une extension à tout schéma généralisé. On se réfère à \cite[7.1.22]{Du} pour la preuve que lorsque l'on pose $X=\widehat{\Spec\Z}$, ces deux définitions en apparence si différentes sont en fait équivalentes:

\begin{definition} Soit $(X,\mathcal{O}_X)$ un schéma généralisé. On définit les \textit{fibrés en droite} sur $X$ comme étant les $\mathcal{O}_X$-modules $L$ localement isomorphes à $\abs{\mathcal{O}_X}=\mathcal{O}_X(1)$. 
\end{definition}

L'ensemble des classes d'isomorphie de fibrés en droite est un groupe pour le produit tensoriel $\otimes_{\mathcal{O}_X}$, l'élément neutre étant $\abs{\mathcal{O}_X}$. Si l'on pose $L^{-1}=\Hom(L,\abs{\mathcal{O}_X})$ pour tout fibré en droite $L$, il existe un isomorphisme $L \otimes_{\mathcal{O}_X} L^{-1} \longrightarrow \abs{\mathcal{O}_X}$, donc $L^{-1}$ est l'inverse de $L$. On appelle \textit{groupe de Picard} d'un schéma généralisé $X$ le groupe de classes d'isomorphie de fibrés en droite sur $X$ par rapport au produit $\otimes_{\mathcal{O}_X}$. L'objectif de ce dernière paragraphe est calculer le groupe de Picard de $\widehat{\Spec \Z}$. Pour faire cela, on remarque d'abord que le $\mathrm{Pic}(X)$ provenant de la définition 23 coïncide avec le groupe de Picard d'un schéma classique. Notamment, pour un anneau de Dedekind usuel $\mathrm{Pic}(\Spec R)=0$ si et seulement si $R$ est factoriel \cite[6.2]{Har}. On en déduit que $\mathrm{Pic}(\Spec\Z)=0$. De même, on peut montrer que tout fibré en droite sur le spectre de l'anneau généralisé $A_N$ est trivial à l'aide du lemme suivant, dont la preuve se trouve dans \cite[7.1.33]{Du}:           

\begin{lemme} Soit $P$ un $A_N$-module projectif de type fini. Si $\dim_\Q P_{(\Q)}=1$, alors $P$ est isomorphe à $\abs{A_N}$.
\end{lemme}

En effet, soit $L \in \mathrm{Pic}(\Spec A_N)$. Il existe un $A_N$-module projectif de type fini $P$ tel que $L=\tilde{P}$. Puisque $L$ est un fibré en droite, la fibre du point générique $L_\xi \cong P_{(\Q)}$ est unidimensionnelle. D'après le lemme $P \cong \abs{A_N}$, donc $L=\widetilde{\abs{A_N}}=\mathcal{O}_{\mathrm{Spec} A_N}$ est le fibré trivial sur $\Spec A_N$. Après avoir montré que tous les fibrés en droite sur les spectres de $\Z$ et de $A_N$ sont triviaux, ce n'est pas difficile à prouver que:    

\begin{theoreme} Le groupe de Picard de la compactification de $\Spec\Z$ est isomorphe au groupe additive des rationnels positifs. En particulier, c'est un groupe abélien libre de rang infini engendré par $\mathcal{O}(\log p)$ lorsque $p$ parcourt l'ensemble des nombres premiers.  
\end{theoreme}

\begin{proof} La compactification de $\Spec\Z$ étant la limite projective des schémas généralisés $\widehat{\Spec\Z}^{(N)}$, la catégorie des fibrés en droite sur $\widehat{\Spec\Z}$ est la limite inductive des catégories des fibrés en droite sur $\widehat{\Spec\Z}^{(N)}$. Par conséquent: 
\begin{equation*}
\mathrm{Pic}(\widehat{\Spec\Z})=\varinjlim_{N>1} \mathrm{Pic}(\widehat{\Spec\Z}^{(N)}). 
\end{equation*} Pour calculer ces groupes, on remarque que le schéma généralisé $\widehat{\Spec\Z}^{(N)}$ est la réunion des ouverts $U_1=\Spec\Z$ et $U_2=\Spec A_N$ le long de $U_1\cap U_2=\Spec B_N$. Puisque les groupes de Picard de $\Spec \Z$ et de $\Spec A_N$ sont triviaux, étant donné un fibré en droite $L$ sur $\widehat{\Spec\Z}^{(N)}$, on peut choisir des trivialisations:
\begin{equation*}
\varphi_1: \mathcal{O}_{\Spec\Z} \longrightarrow L_{|U_1}, \quad \varphi_1: \mathcal{O}_{\Spec A_N} \longrightarrow L_{|U_2}
\end{equation*} Ce choix est canonique quitte à fixer le signe de $\varphi_1$ et $\varphi_2$, car l'on peut toujours multiplier les trivialisations par des éléments inversibles $s_1\in \Gamma(U_1,\mathcal{O}_{U_1}^\times)=\Z^\times=\{\pm 1\}$ et $s_2\in \Gamma(U_2,\mathcal{O}_{U_2}^\times)=\abs{A_N}^\times=\{\pm 1 \}$. Sur l'intersection $U_1\cap U_2$ ces trivialisations vérifient la condition de compatibilité ${\varphi_2}_{|U_1\cap U_2}=\lambda {\varphi_1}_{|U_1\cap U_2}$ pour un élément $\lambda$ inversible dans $B_N$. Réciproquement, si l'on se donne $\lambda\in B_N^\times$ positif, on peut construire un fibré en droite $L=\mathcal{O}(\log \lambda)$ en recollant des fibrés en droite triviaux sur $U_1$ et sur $U_2$ selon la formule ci-dessus. Lorsque le signe de $\lambda$ est fixé, cette correspondance est bijective. Donc, $\mathrm{Pic}(\widehat{\Spec\Z}^{N})$ est isomorphe à $\log B_N^\times/\{\pm 1\}$, c'est-à-dire, au group abélien $\log B_{N,+}^\times$ des éléments positifs inversibles de l'anneau $B_N=\Z[N^{-1}]$. Si l'on note par $p_1,\ldots,p_r$ l'ensemble des diviseurs premiers de $N$, le groupe de Picard est un groupe abélien libre de rang r ayant base $\{\mathcal{O}(\log p_i\}_{1\leq i\leq r}$. On en déduit:
\begin{equation*}
\mathrm{Pic}(\widehat{\Spec\Z})=\varinjlim_{N>1} \log B_{N,+}^\times=\log \Q_+^\ast. \qedhere
\end{equation*}
\end{proof}

\begin{remarque} Le fait que le rang du groupe de Picard de $\widehat{\Spec\Z}^{(N)}$ soit le nombre de diviseurs premiers de $N$ fournit une preuve alternative de l'énoncé affirmant que $f_M^N$ n'est pas un isomorphisme lorsqu'il existe $p$ premier tel que $p|M$, mais $p\nmid N$. 
\end{remarque}

\section{Questions ouvertes}

\begin{enumerate}
\item Comme on a remarqué dans la section concernant l'analogie entre les corps de nombres et les corps de fonctions, c'est nécessaire d'introduire un facteur local à l'infini afin d'avoir une équation fonctionnelle pour la fonction zêta. L'expression intégrale de ce nouveau facteur est tout à fait analogue à celle des facteurs correspondant aux premiers finis sauf pour l'apparition de la gaussienne $e^{-\pi x^2}$ au lieu de la fonction indicatrice de $\Z_\infty$, c'est-à-dire, de l'intervalle $[-1,1]$, que l'on s'attendrait naivement. L'intervention de la gaussienne ne peut pas être sans relation avec une interprétation probabiliste du passage à l'infini. À l'avis de l'auteur, bien comprendre ce facteur local à l'infini mérite toute notre attention, étant donné que la preuve de l'équation fonctionnelle des fonctions zêta des variétés de dimension supérieure sur $\Z$ reste toujours une question ouverte.   

%\begin{figure}
  %\centering
   % \includegraphics[scale=0.6]{f_gaussian_7}
  %\caption{La fonction $e^{-\pi x^2}$ et l'indicatrice de $\Z_\infty$}
  %\label{gauss}
%\end{figure}

\item Dans ce mémoire, on ne s'est pas occupé de la question si les groupes linéaires algébriques sont définis (au sens de Durov) sur $\F_1$ ou sur une extension convenable. En termes de foncteurs représentables \cite[5.1.21]{Du} et puis en introduisant la notion de déterminant, Durov montre \cite[5.5.17]{Du} que les seules schémas affines en groupes définis sur $\F_\emptyset$ sont les tores $\mathbb{G}^n_m$, que l'on peut définir le groupe général linéaire sur $\F_1$ et que $\SL_n$ existe sur l'extension quadratique $\F_{1^2}$. Cependant, on ne sait pas si l'on peut définir, par exemple, le groupe orthogonal dans ce cadre, car le concept de transposée d'une matrice pose des problèmes quand on travaille dans l'algèbre non-additive. Notamment, l'auteur s'intéresse à une possible application du concept récent de variété torifiée introduit dans \cite{LoLo} à la recherche de modèles d'autres groupes sur $\F_1$ ou $\F_{1^2}$. Dans ce même papier, J. López Peña et O. Lorscheid abordent la comparaison des notions de schéma sur le corps à un élément de Deitmar, Soulé et Connes-Consani. Il serait nécessaire d'étendre cette comparaison aux géométries de Töen-Vaquié, Shai-Haran et Durov.  

\smallskip

\item L'espace projectif $\mathbb{P}^{n-1}(\F_1)$ construit par Durov en termes du spectre projectif d'un anneau généralisé gradué ne contient pas $n$ points comme l'on s'attendait, mais $2^n-1$. Cependant, étant donné que $2^n-1$ est le nombre de parties non-vides d'un ensemble de $n$ points, l'auteur pense que c'est possible de trouver une certaine relation d'équivalence ou opération $n$-aire à $n$ orbites fermées. Notamment, il peut être intéressant de comparer le $\mathbb{P}^{n-1}(\F_1)$ de Durov avec l'espaces projectifs définis par Connes-Consani dans \cite{CoCa} comme le foncteur gradué
\begin{equation*}
\mathbb{P}^d: \mathcal{F}_{ab} \longrightarrow \mathfrak{Ens}, \quad \mathbb{P}^d(D)^{(k)}=\coprod_{Y\subset \{1,\ldots,d+1\}, \ \abs{Y}=k+1} D^Y/D,
\end{equation*} où $D$ agit à droite sur $D^{Y}$ par l'action diagonale. Ainsi, la partie de degré nul sur $\F_{1^n}$ est juste $\{1,\ldots,d+1\}$, donc $\abs{\mathbb{P}^{n-1}(\F_{1^n})^{(0)}}=n$. D'après la définition de Connes-Consani, la structure du projectif sur le corps à un élément est plus compliquée que celle d'un ensemble fini, mais la partie de degré nul coïncide avec la prévision de Tits. On peut attendre que cela soit pareil après une légère modification de la construction de Durov. 

\smallskip

\item Comme on a déjà signalé, le fait que le produit tensoriel $\Z \otimes_{\F_1} \Z$ soit $\Z$ de nouveau, et donc $\Spec\Z \times_{\Spec\F_1} \Spec\Z=\Spec\Z$, est en quelque mode décevant, car une des motivations initiales pour développer toute la théorie des anneaux et des schémas généralisés était la recherche d'une catégorie ayant des produits arithmétiques.
D'après Durov, cela rend inutile sur son approche en ce qui concerne l'étude de la fonction zêta. De l'autre côte, Connes-Consani, inspirés des idées de Soulé, ont réussi à construire les fonctions zêta de schémas sur $\F_1$. La compactification de $\Spec\Z$ obtenu dans ce mémoire a des très bonnes propriétés topologiques (tous les points sur $\xi$ sont fermés comme dans une courbe projective), mais par exemple son groupe de Picard n'est pas $\R_{>0}$ comme l'analogie avec le cas des corps de fonctions suggère, mais le groupe additif des rationnels positifs. Une façon \textit{definitive} de tester que l'on a trouvé la compactification correcte serait construire un cadre suffisamment vaste pour définir les fonctions zêta des schémas au sens de Durov et vérifier si cela donne 
\begin{equation*}
Z(\widehat{\Spec\Z},\ s)=\xi(s)=\pi^{-\frac{s}{2}}\Gamma(s/2)\zeta(s).
\end{equation*} En quelque sens, on peut dire que Durov a la courbe et Connes et Consani ont la fonction. À l'avis de l'auteur, la topologie de $\widehat{\Spec \Z}$ mérite encore plus d'étude. En effet, puisque la restriction à $\Spec \Z$ de chacun des morphismes $f_N^M$ du système projectif définissant la compactification est un isomorphisme, $f_N^M$ n'agit que sur le point à l'infini. Étant donné que le résultat final est un modèle lisse de $\Q$, cette limite inductive est une sorte de résolution infinie de singularités. Une autre question qui se pose est le calcul du produit tensoriel $\Z_{\infty} \otimes_{\F_1} \Z_{\infty}$.        
\end{enumerate}

\appendix

\section{Construction fonctorielle des schémas sur $\F_1$} 

Dans cet annexe, on présente une définition fonctorielle des schémas généralisés par analogie avec la géométrie algébrique classique. Rappelons d'abord qu'un \textit{faisceau} $\mathcal{F}$ sur un espace topologique $X$ est la donnée 
\begin{enumerate}
\item d'un ensemble (peut-être muni d'une certaine structure algébrique additionnelle) $\mathcal{F}(U)$, que l'on appelle les sections du faisceau, défini pour chaque ouvert $U \subset X$
\item d'une application de restriction des sections $p_{UV}: \mathcal{F}(U) \longrightarrow \mathcal{F}(V)$ définie pour chaque inclusion $V \subset U$, que l'on note $f \longmapsto f_{|V}$. 
\end{enumerate} Ensuite, on requiert la condition de recollement suivante: pour chaque recouvrement ouvert $\{U_i\}_{i\in I}$ de $U$, si $f_i \in \mathcal{F}(U_i)$ sont des sections dont les restrictions 
\begin{equation*}
(f_j)_{|(U_i\cap U_j)}=(f_i)_{|(U_i\cap U_j)}
\end{equation*} aux intersections des $U_i$ coïncident, il existe une unique section $f \in \mathcal{F}(U)$ dont les restrictions $f_{|U_i}=f_i$ aux $U_i$ sont égales aux sections données au départ. 

\smallskip

Essayons de récrire cette définition en termes catégoriques. Étant donné l'espace topologique $X$, on défini une catégorie $\mathrm{Ouv}(X)$ dont les objets sont les ouverts de $X$ et dont les morphismes sont les inclusions. Alors un faisceau $\mathcal{F}$ à valeurs dans une catégorie $\mathcal{C}$ basée sur les ensembles est un foncteur contravariant $\mathcal{F}: \mathrm{Ouv}(X)^\text{op} \longrightarrow \mathcal{C}$, où $\mathrm{Ouv}(X)^\text{op}$ représente, comme d'habitude, la catégorie duale$^\star$ de $\mathrm{Ouv}(X)$. Si $f: X \longrightarrow Y$ est une application continue entre les espaces topologiques $X$ et $Y$ et $\mathcal{F}$ est un faisceau sur $X$, on définit le faisceau $f_\ast\mathcal{F}$ sur $Y$ par la formule $f_\ast\mathcal{F}(V)=\mathcal{F}(f^{-1}(V))$ pour tout ouvert $V$ de $Y$. Maintenant, la condition de recollement signifie qu'une section globale est déterminée par ses valeurs locales et que si l'on a des sections locales qui se recollent, on peut construire une section globale. Cela se traduit dans l'exactitude de la suite
\begin{equation*}
\mathcal{F}(U) \longrightarrow \prod_{i \in I} \mathcal{F}(U_i) \rightrightarrows \prod_{i\in I} \mathcal{F}(U_i \cap U_j),
\end{equation*} où les flèches sont donnés par des restrictions des sections, pour tout recouvrement ouvert $\{U_i\}_{i\in I}$ de $U$. En effet, cela signifie que $\mathcal{F}(U)$ s'injecte dans le produit $\prod_{i \in I} \mathcal{F}(U_i)$ et que $\mathcal{F}(U)$ est l'égaliseur$^\star$ du couple des flèches $\prod_{i \in I} \mathcal{F}(U_i) \rightrightarrows \prod_{i\in I} \mathcal{F}(U_i \cap U_j)$, autrement dit, que si l'on a des $f_i \in \mathcal{F}(U_i)$ dont les images par les deux restrictions possibles à $U_i \cap U_j$ coïncident, il existe une seule section $f\in \mathcal{F}(U)$ telle que $f_{|U_i}=f_i$. Notons que dans la catégorie des espaces topologiques, $U_i	\cap U_j$ est isomorphe à $U_i \times_U U_j$.  

\smallskip

Pour généraliser la notion de faisceau à une catégorie quelconque, on a d'abord besoin de définir ce que l'on appelle une topologie de Grothendieck, une notion qui étend de façon naturelle les propriétés des recouvrements ouverts dans un espace topologique. Pour notre but, il suffit de définir une structure un peu plus générale:  

\begin{definition} Une \textit{prétopologie de Grothendieck} $\mathcal{T}$ dans une catégorie $\mathcal{C}$ est la donnée d'un ensemble $\mathrm{Rec}(\mathcal{T})$ de familles de morphismes (dites recouvrements) $\{\varphi_i: U_i \longrightarrow U\}_{i\in I}$ dans $\mathcal{C}$ définies pour chaque objet $U$ et telles que:
\begin{enumerate}
\item Si $\varphi$ est un isomorphisme, la famille $\{\varphi\}$ appartient à $\mathrm{Rec}(\mathcal{T})$.  
\item Si $\{\varphi_i: U_i \longrightarrow U\}$ et $\{\psi_{ij}: V_{ij} \longrightarrow U_i\}$ sont deux recouvrements dans $\mathcal{T}$, alors le recouvrement composé $\{ \varphi_i \circ \psi_{ij}: V_{ij} \longrightarrow U \}$ est de nouveau dans $\mathrm{Rec}(\mathcal{T})$.  
\item Si $\{U_i \longrightarrow U \}$ est un recouvrement et $V \longrightarrow U$ est un morphisme quelconque, le produit fibré $U_i \times_U V$ existe et la famille $\{U_i \times_U V \longrightarrow V\}$ appartient à $\mathrm{Rec}(\mathcal{T})$. 
\end{enumerate} On dit qu'une catégorie $\mathcal{C}$ muni d'une prétopologie de Grothendieck est un \textit{prétopos}. 
\end{definition}

Notamment, la catégorie $\mathrm{Ouv}(X)$ définie au début est un prétopos dont les familles recouvrantes sont les données par des inclusions. Les prétopos constituent, en quelque sens, le cadre minimal où l'on peut définir la notion de faisceau.  

\begin{definition} Soit $(\mathcal{C}, \mathcal{T})$ un prétopos et soit $\mathcal{D}$ une catégorie ayant des produits. On appelle \textit{faisceau} sur $\mathcal{C}$ à valeurs dans $\mathcal{D}$ un foncteur $F: \mathcal{C}^\text{op} \longrightarrow \mathcal{D}$ tel que pour tout famille recouvrante $\{\varphi_i: U_i\longrightarrow U \}$, la suite
\begin{equation*}
F(U) \longrightarrow \prod F(U_i) \rightrightarrows \prod F(U_i \times_U U_j),
\end{equation*} où toutes les flèches sont données par des restrictions des sections, soit exacte. 
\end{definition}

Cette définition va nous permettre de construire des faisceaux sur les schémas affines généralisés. Avant de les définir, on montre que c'est naturel de définir la catégorie des schémas affines généralisés comme la catégorie duale$^\star$ de celle des anneaux généralisés en rappelant le cas classique. Soit $A$ un anneau commutatif. La topologie de Zariski muni l'ensemble d'idéaux premiers de $A$ d'une structure d'espace topologique, que l'on note $\Spec A$. On veut définir maintenant un faisceau d'anneaux sur $\Spec A$ en imitant les propriétés des fonctions régulières sur une variété. Soit $X\subset \mathbb{A}^n_k$ une variété algébrique affine définie sur un corps $k$. Une fonction $f: X \longrightarrow k$ est dite régulière en un point $P \in X$ s'il existe un voisinage ouvert de $P$ dans $X$ tel que $f$ est le quotient de deux polynômes $g, h \in k[X_1,\ldots,X_n]$ dont $h$ ne s'annule pas sur $U$. Si l'on note par $\mathcal{O}(U)$ l'anneau des fonctions régulières sur $U$, on obtient ce que l'on appelle le faisceau structurel de $X$. Dans le cas du spectre, il n'y a pas un corps de base distingué comme but des fonctions, ce qui nous amène à prendre en même temps toutes les localisations $A_\mathfrak{p}$, où $\mathfrak{p} \in \Spec A$. Alors, étant donné un ouvert $U \subset \Spec A$, on considère les fonctions $s: U \longrightarrow   \coprod_{\mathfrak{p}\in U} A_\mathfrak{p}$ satisfaisant la condition de régularité suivante: soit $\mathfrak{p}\in U$, on requiert qu'il existe un voisinage ouvert $\mathfrak{p} \in V \subset U$ et deux éléments $a,f\in A$ tels que $f\notin \mathfrak{q}$ pour tout $\mathfrak{q}\in V$ (l'analogue de la non-annulation du polynôme $h$) et que $s(\mathfrak{q})=a/f$. L'ensemble $\mathcal{O}(U)$ formé des fonctions vérifiant la condition ci-dessus pour tout $\mathfrak{p}\in U$ admet une structure d'anneau, et l'on obtient ainsi un faisceau d'anneaux $\mathcal{O}$ sur $\Spec A$.    

\smallskip

Cette construction associe à chaque anneau commutatif $A$ un couple $(\Spec A, \mathcal{O})$ formé d'un espace topologique et d'un faisceau sur lui. Une question qui se pose immédiatement est s'il existe une catégorie $\mathcal{C}$ tel que $(\Spec A, \mathcal{O})$ soit un objet dans $\mathcal{C}$ pour tout anneau commutatif $A$ et que la correspondance $A \longmapsto (\Spec A, \mathcal{O})$ soit fonctorielle. Ce sont les \textit{espaces localement annélés}. Les objets de la catégorie sont les couples $(X, \mathcal{O}_X)$ formés d'un espace topologique $X$ et d'un faisceau d'anneaux sur $X$ tel que l'anneau des germes
\begin{equation*}
\mathcal{O}_{X,P}=\varinjlim_{P \in U}\mathcal{O}_P^\bullet, \quad \text{où} \quad \mathcal{O}_P^\bullet=((\mathcal{O}(U))_{P \in U}, (p_{UV})_{V\subset U})
\end{equation*} soit local pour tout $P \in X$. Les morphismes entre deux espaces localement annélés $(X, \mathcal{O}_X)$ et $(Y,\mathcal{O}_Y)$ sont les couples $(f,f^\sharp)$, où $f: X \longrightarrow Y$ est une application continue et $f^\sharp: \mathcal{O}_Y \longrightarrow f_\ast \mathcal{O}_X$ est un morphisme de faisceaux telle que l'application induite $f^\sharp: \mathcal{O}_{Y,f(P)} \longrightarrow \mathcal{O}_{X,P}$ soit un homomorphisme local d'anneaux locaux pour tout $P\in X$. 

\smallskip

Un \textit{schéma affine} est un espace localement annélé qui est isomorphe à $(\Spec A, \mathcal{O})$ pour un certain anneau commutatif $A$. On obtient ainsi une application des anneaux commutatifs vers les schémas affines, qui est fonctorielle au sens suivant: pour chaque homomorphisme d'anneaux $\varphi: A \longrightarrow B$, la formule $f(\mathfrak{p})=\varphi^{-1}(\mathfrak{p})$ définit une application continue $f:=\mathrm{Spec}(\varphi):  \Spec B \longrightarrow \Spec A$. De même, en localisant $\varphi$, on obtient un homomorphisme local $\varphi_\mathfrak{p}: A_{\varphi^{-1}(\mathfrak{p})} \longrightarrow B_\mathfrak{p}$ qui peut s'étendre, par la construction du faisceau structurel d'un spectre, à un morphisme de faisceaux $f^\sharp: \mathcal{O}_{\mathrm{Spec}A} \longrightarrow \mathcal{O}_{\mathrm{Spec}B}$ dont la restriction à chaque anneau de germes est $\varphi_\mathfrak{p}$. Par conséquent, $(f,f^\sharp)$ est un morphisme d'espaces localement annelés. On en déduit que le foncteur $\mathrm{Spec}$ est une équivalence contravariante entre la catégorie des anneaux commutatifs et la catégorie des schémas affines, donc \textit{les schémas affines constituent la catégorie duale des anneaux commutatifs}. Finalement, un \textit{schéma} est un espace localement annélé $(X, \mathcal{O}_X)$ où chaque point $P$ admet un voisinage ouvert $U$ dans $X$ tel que $(U, \mathcal{O}_{X|U})$ soit un schéma affine. 

\smallskip

Soit maintenant $\mathfrak{Gen}$ la catégorie des anneaux généralisés. D'après la discussion précédente, on définit les schémas affines généralisés comme la catégorie duale de $\mathfrak{Gen}$, i.e. $\mathfrak{Aff}:=\mathfrak{Gen}^{\textit{op}}.$ On note $\Spec R$, où $R$ est un anneau généralisé, les objets de cette catégorie. En particulier, un schéma affine généralisé est un foncteur contravariant de $\mathfrak{Gen}$ vers les ensembles. Étant donné un schéma affine généralisé $\Spec R$, pour définir une prétopologie de Grothendieck on considère les localisations de l'anneau généralisé $R$ dans des éléments $f_i \in \abs{R}$. Si l'on pose $U=\Spec R$ et $U_i=\Spec R_{f_i}$, on définit les familles recouvrantes comme les $\{U_i \longrightarrow U\}_{i\in I}$ telles que $\coprod_{i\in I} U_i \longrightarrow U$ soit une surjection au niveau d'espaces topologiques. Ce n'est pas difficile à vérifier que ces données satisfont les axiomes d'une prétopologie de Grothendieck, que l'on appellera la topologie de Zariski.  
Par conséquent, un faisceau sur le prétopos $\mathfrak{Aff}$ est un foncteur $F: \mathfrak{Aff}^\text{op}=\mathfrak{Gen} \longrightarrow \mathfrak{Ens}$ telle que la suite
\begin{equation*}
F(U) \longrightarrow \prod F(U_i) \rightrightarrows \prod F(U_i \times_U U_j)
\end{equation*} soit exacte pour toute famille recouvrante $\{U_i \longrightarrow U\}$ dans $\mathfrak{Aff}$.    

\smallskip

Il nous reste seulement à définir une condition analogue au fait qu'un schéma soit localement affine. On remarque que si l'on voit le spectre d'un anneau commutatif classique $R$ comme le foncteur de $\mathfrak{Ann}$ vers $\mathfrak{Ens}$ qu'au niveau des objets vaut $A \longrightarrow \Spec R(A)=\Hom_{\mathfrak{Ann}}(R,A)$, les ouverts de Zariski sont alors de sous-foncteurs de $\Spec R$ vérifiant quelques propriétés liées à la localisation. Cette notion se formalise de la façon suivante: 
       
\begin{definition} Soit $S: \mathfrak{Gen} \longrightarrow \mathfrak{Ens}$ un foncteur contravariant. Un sous-foncteur $U$ de $S$ est dit \textit{ouvert de Zariski} si pour tout anneau généralisé $B$ et pour tout transformation naturelle $u: \Spec B \longrightarrow S$, il existe un système multiplicatif $T \subset B$ tel que le produit fibré $U \times_S \Spec B$ soit naturellement isomorphe à $\Spec B[T^{-1}]$.
\end{definition}

Alors, un schéma généralisé est un faisceau qui admet un recouvrement par des ouverts isomorphes à des schémas affines généralisés. Plus précisément:

\begin{definition} Un \textit{schéma généralisé} est un foncteur contravariant $S: \mathfrak{Gen} \longrightarrow \mathfrak{Ens}$ qui est un faisceau pour la topologie de Zariski et qui est localement représentable, au sens où il existe un recouvrement $\{S_i\}_{i\in I}$ de $S$ par des sous-foncteurs ouverts de Zariski qui sont naturellement isomorphes à des schémas affines généralisés, c'est-à-dire, tels que $S_i\cong \Spec R_i$ pour des anneaux généralisés $R_i$.
\end{definition}

%\begin{table}[ht]
%\caption{Schémas classiques vers schémas généralisés}
%\centering
%\begin{tabular}{c c}
%\\[0.5ex] \hline 
%Géométrie algébrique classique & Géométrie algébrique sur $\F_1$ \\
%\hline 
%& \\
%Anneaux commutatifs $\mathfrak{Ann}$ & Monades algébriques \\[0.3ex]
%& commutatives $\mathfrak{Gen}$ \\[0.3ex]
%Idéaux premiers $\mathfrak{p}$ de $\Spec A$ & Sous-modules $\mathfrak{p}$ de %$\abs{R}$ tels \\[0.3ex]
%& que $\abs{R}-\mathfrak{p}$ est multiplicatif \\[0.3ex]
%\hline 
%\end{tabular}
%\end{table}

\section{Annexe: quelques outils catégoriques}

\textbf{Catégorie duale} Soit $\mathcal{C}$ une catégorie. On appelle catégorie duale de $\mathcal{C}$ la catégorie $\mathcal{C}^\text{op}$ dont les objets coïncident avec ceux de $\mathcal{C}$ et dont les morphismes sont obtenus en inversant les buts et sources des morphismes de $\mathcal{C}$, autrement dit $\Hom_{\mathcal{C}^\text{op}}(X,Y)\cong \Hom_\mathcal{C}(Y,X)$ pour tout couple d'objets $X,Y$ dans $\mathcal{C}$. La notion de catégorie duale permet par exemple d'éliminer la distinction entre foncteur covariant et contravariant, car un foncteur contravariant $F: \mathcal{C} \longrightarrow \mathcal{D}$ est un foncteur covariant $F: \mathcal{C}^\text{op}$ sur la catégorie duale. 

\medskip

\textbf{Catégorie monoïdale.} Une catégorie $\mathcal{C}$ est munie d'une structure de \textit{catégorie monoïdale} (ou de $\otimes$-catégorie associative unitaire) s'il existe 
\begin{enumerate}
\item un bifoncteur $\otimes: \mathcal{C} \times \mathcal{C} \longrightarrow \mathcal{C}$ associative à isomorphisme naturel près (i.e. $X\otimes(Y\otimes Z) \cong (X\otimes Y)\otimes Z$ pour tout triplet d'objets $(X,Y,Z)$)
\item une unité $\bf{1}\in\text{Ob(}\mathcal{C})$ à isomorphisme naturel près (i.e. ${\bf{1}}\otimes X \cong X \cong X\otimes {\bf{1}}$ pour tout objet $X$) 
\end{enumerate} tels que $\otimes$ vérifie le diagramme du pentagone et que $\bf{e}$ soit compatible avec l'associativité (voir \cite[XI.1]{McLa}). Des exemples de catégories monoïdales sont les ensembles avec le produit cartésien $(\mathfrak{Ens},\times,\{1\})$ ou les $R-$modules avec le produit tensoriel $(R-\mathfrak{Mod},\otimes,R).$ On dit qu'une catégorie monoïdale est symétrique s'il y a un isomorphisme naturel entre $X\otimes Y$ et $Y \otimes X$ pour tout $X,Y$ dans $\mathcal{C}.$

\medskip

\textbf{(Co)égaliseur.} Soient $X$ et $Y$ deux objets d'une catégorie $\mathcal{C}$ et soient $f,g: X \rightrightarrows Y$ deux morphismes de $X$ vers $Y.$ On appelle égaliseur du système le seul couple (à isomorphisme près) $(E, e)$ formé d'un objet $E$ de $\mathcal{C}$ et d'un morphisme $u: E \longrightarrow X$ tels que $f\circ e=g\circ e$ et que 
\begin{equation*}
\Hom_\mathcal{C}(P, E)\cong \{i\in \Hom_\mathcal{C}(P,X): f\circ i=g\circ i\}
\end{equation*}
soient naturellement isomorphes pour tout $P$ dans $\mathcal{C}.$ On appelle coégaliseur du système le seul couple (à isomorphisme près) $(Q,q)$ formé d'un objet de $\mathcal{C}$ et d'un morphisme $q: Y\longrightarrow Q$ tels que $q\circ f=q\circ g$ et que
\begin{equation*}
\Hom_\mathcal{C}(Q,P)\cong \{i\in \Hom_\mathcal{C}(Y,P): i\circ f=i\circ g
\end{equation*} soient naturellement isomorphes pour tout $P$ dans $\mathcal{C}.$ Autrement dit, l'égaliseur et le coégaliseur ont la propriété universelle de rendre les flèches $f$ et $g$ égales.  

\medskip

\textbf{Composition de transformations naturelles.} Soient $\mathcal{C},\mathcal{D},\mathcal{E}$ des catégories, $F, F':\mathcal{C} \longrightarrow \mathcal{D}$ et $G,G':\mathcal{D} \longrightarrow \mathcal{E}$ des foncteurs et $\xi: F \longrightarrow F', \eta: G \longrightarrow G'$ des transformations naturelles. On définit $\eta \star \xi: GF \longrightarrow G'F'$ comme la transformation naturelle dont les composantes sont 
\begin{equation*}
(\eta \star \xi)_X=\eta_{F'(X)}\circ G(\xi_X)=G'(\xi_X)\circ \eta_{F(X)}
\end{equation*} pour chaque objet $X$ dans $\mathcal{C}.$ Si $F$ est un foncteur, on pose $\eta\star F:=\eta\star \id_F.$

\medskip

\textbf{Coproduit.} Soit $\mathcal{C}$ une catégorie et soient $X$ et $Y$ deux objets de $\mathcal{C}.$ On appelle \textit{coproduit} (ou somme catégorique) de $X$ et $Y$ le seul objet $X\amalg Y$ tel que pour tout $Z \in \mathrm{Ob}(\mathcal{C})$ il existe un isomorphisme naturel 
\begin{equation*}
\Hom(X\amalg Y, Z) \cong \Hom(X,Z) \times \Hom(Y,Z). 
\end{equation*} Par exemple, dans la catégorie des ensembles $\mathfrak{Ens}$, la somme est la réunion disjointe, dans la catégorie des espaces topologiques pointés $\mathfrak{Top}_\ast$, c'est le produit wedge, dans la catégorie des groupes $\mathfrak{Grp}$ (resp. abéliens), c'est le produit libre (resp. la somme directe), et dans la catégorie des anneaux commutatifs $\mathfrak{Ann}$, c'est le produit tensoriel.

\medskip

\textbf{Foncteurs adjoints.} Soient $\mathcal{C}$ et $\mathcal{D}$ deux catégories. On dit que les foncteurs $F: \mathcal{C} \longrightarrow \mathcal{D}$ et $G: \mathcal{D} \longrightarrow \mathcal{C}$ sont \textit{adjoints} si et seulement si 
\begin{equation*}
\Hom_\mathcal{C}(X, G(Y))\cong \Hom_\mathcal{D}(F(X),Y),
\end{equation*} pour tout couple d'objets $X\in\text{Ob(}\mathcal{C}), Y\in\text{Ob(}\mathcal{D}),$ l'isomorphisme ci-dessus étant naturel en $X$ et $Y.$ Dans ce cas, le foncteur $F$ (resp. $G$) est dit l'adjoint à gauche (resp. à droite) du foncteur $G$ (resp. $F$). Lorsque l'on pose $Y=F(X)$ dans la définition, $\id_{F(X)}$ appartient au deuxième ensemble. Son image par l'isomorphisme est une application $\eta_X: X\longrightarrow GF(X)$ naturelle en $X.$ On appelle \textit{unité de l'adjonction} la transformation naturelle $\eta: \id_\mathcal{C} \longrightarrow GF$ dont les composantes sont données par $\eta_X.$ De même, une adjonction admet une \textit{co-unité} $\xi: FG \longrightarrow \id_\mathcal{D}$ définie de manière analogue en prenant $X=G(Y)$ dans la formule.     

\medskip

\textbf{Foncteur (pleinement) fidèle.} Considérons deux catégories $\mathcal{C}$ et $\mathcal{D}$ et un foncteur covariant $F: \mathcal{C} \longrightarrow \mathcal{D}$. Comme d'habitude, étant donné un couple d'objets $(X,Y)$ dans $\mathcal{C}$, F induit une application
\begin{equation*}
F_{X,Y}: \Hom_\mathcal{C}(X,Y) \longrightarrow \Hom_\mathcal{D}(F(X),F(Y))
\end{equation*} On dit alors que F est fidèle (resp. plein, pleinement fidèle) si $F_{X,Y}$ est injective (resp. surjective, bijective) pour tout $X,Y$ dans $\mathcal{C}.$  

\medskip

\textbf{Lemme de Yoneda.} Soient $X$ et $Y$ deux objets dans une catégorie $\mathcal{C}.$ Notons $h_A:=\Hom_\mathcal{C}(A,\cdot)$ le foncteur représenté par un objet $A$. Avec ces notations, le lemme de Yoneda affirme que pour toute transformation naturelle $\eta: h_Y \longrightarrow h_X,$ il existe un morphisme $f: X \longrightarrow Y$ tel que $\eta=h_f,$ où $h_f: \Hom_\mathcal{C}(Y,\cdot) \longrightarrow \Hom_\mathcal{C}(X,\cdot)$ est la transformation naturelle ayant composantes $h_{f,A}$ qui appliquent $\varphi: Y \longrightarrow A$ dans $\varphi \circ f: X \longrightarrow Y$. Notamment, cela implique que les objets représentant des foncteurs isomorphes sont ils mêmes isomorphes.   

\medskip

\textbf{Limite inductive.} Soit $\mathcal{C}$ une catégorie et soit $(I,\leq)$ un ensemble partiellement ordonné. Un système inductive d'objets dans $\mathcal{C}$ indexé par $I$ est une famille
\begin{equation*}
A_\bullet=((A_i)_{i\in I}, (f_{i,j})_{i\leq j})
\end{equation*} d'objets et pour chaque $i\leq j$ de morphismes $f_{i,j} : A_i \longrightarrow A_j$ tels que $f_{i,i}=\id_{A_i}$ et que $f_{i,k} = f_{i,j}\circ  f_{j,k}.$ Une limite inductive pour $A_\bullet$ est alors un objet $\varinjlim_{I}A_\bullet$ tel que pour tout objet $Z$ dans $\mathcal{C},$ on ait une bijection naturelle
\begin{equation*}
\Hom(\varinjlim_{I}A_\bullet, Z) \cong \varprojlim_{I}\Hom(A_i, Z),
\end{equation*} cette dernière limite étant la famille incluse dans $\prod_i \Hom (Z,A_i)$ de morphismes $h_i$ tels que $f_{i,j}\circ h_j=h_i.$ On dit que la limite est filtrée lorsque pour chaque couple $(p,q)$ d'éléments de $I,$ il existe $r \in I$ tel que $p\leq r, q\leq r.$  

\medskip

\textbf{Limite projective} Soit $\mathcal{C}$ une catégorie et soit $(I,\leq)$ un ensemble partiellement ordonné. On définit les limites projectives en inversant les buts et sources des morphismes dans la définition de limite inductive. Ainsi un système projectif d'objets dans $\mathcal{C}$ indexé par $I$ est une famille $A_\bullet=((A_i)_{i\in I}, (f_{i,j})_{i\leq j})$ d'objets et pour chaque $i\leq j$ de morphismes $f_{i,j} : A_j \longrightarrow A_i$ tels que $f_{i,i}=\id_{A_i}$ et que $f_{i,k} = f_{i,j}\circ  f_{j,k}.$ Une limite projective pour $A_\bullet$ est un objet $\varprojlim_{I}A_\bullet$ tel que pour tout $Z$ dans $\mathcal{C},$ on ait une bijection naturelle
\begin{equation*}
\Hom(Z, \varprojlim_{I}A_\bullet) \cong \varprojlim_{I}\Hom(Z, A_i),
\end{equation*} cette dernière limite étant l'ensemble des $h_i \in \prod_i \Hom (A_i,Z)$ tels que $f_{i,j}\circ h_j=h_i.$ 

\medskip

\textbf{Objet (co)simplicial.} Soit $\Delta$ la catégorie des ensembles finis totalement ordonnés munis des applications croissantes. Elle est équivalente à la catégorie dont les objets sont les ensembles $\textbf{n}$ et dont les morphismes sont les applications croissantes entre eux. On appelle objet (co)simplicial dans une catégorie $\mathcal{C}$ un foncteur (co)contravariant de $\Delta$ dans $\mathcal{C}.$ D'après l'équivalence des catégories mentionnées, un objet (co)simplical peut être pensé comme une famille d'objets $X_\textbf{n}$ munis des morphismes entre eux correspondant aux applications croissantes $\varphi: \textbf{n} \longrightarrow \textbf{m}$. En ce sens, une monade algébrique est un objet cosimplicial dans la catégorie des ensembles: la famille d'objets est $\{\Sigma(\textbf{n})\}_{n\geq 0}$ et les morphismes sont les applications $\Sigma(\varphi): \Sigma(\textbf{n}) \longrightarrow \Sigma(\textbf{m})$ induites par $\varphi: \textbf{n} \longrightarrow \textbf{m}$.

\medskip

\textbf{Produit fibré.} Soient $X, Y, S$ trois objets d'une catégorie $\mathcal{C}$ ayant des morphismes $f: X \longrightarrow S$ et $g: Y \longrightarrow S$ avec le même but. On appelle produit fibré de $X$ et $Y$ sur $S$ le seul triplet (à isomorphisme près) $(X\times_S Y, p, q)$ formé d'un objet $X\times_S Y$ et de deux morphismes $p: X\times_S Y \longrightarrow X, q: X\times_S Y \longrightarrow Y$  tels que pour tout objet $Z$ dans $\mathcal{C},$ l'on ait une bijection naturel d'ensembles
\begin{equation*}
\Hom(Z,X\times_S Y)\cong\{(h, k) \in \Hom(Z,X)\times \Hom(Z,Y): f\circ h=g\circ k\}.
\end{equation*} Par exemple, dans la catégorie des ensembles, le produit fibré n'est autre chose que $X\times_Z Y=\{(x,y)\in X\times Y: f(x)=g(y)\}$ muni des restrictions des projections usuelles $\text{pr}_1$ et $\text{pr}_2$ à cet ensemble. 

\medskip

\textbf{Sous-catégorie pleine.} Soit $\mathcal{C}$ une catégorie. Une sous-catégorie $\mathcal{D}$ de $\mathcal{C}$ est la donnée de quelques objets et de quelques morphismes dans $\mathcal{C}$ tels que 
\begin{enumerate}
\item Si $f:X \longrightarrow Y$ est un morphisme dans $\mathcal{D},$ alors $X,Y\in \mathrm{Ob}(\mathcal{D}).$
\item Si $X$ est un objet dans $\mathcal{D},$ alors $\id_X$ appartient aux morphismes de $\mathcal{D}.$
\item Si $f,g$ sont des morphismes composables dans $\mathcal{D},$ alors leur composition appartient à $\mathcal{D}.$
\end{enumerate} On dit que $\mathcal{D}$ est \textit{pleine} si le foncteur d'inclusion $i: \mathcal{D}\longrightarrow \mathcal{C}$ est plein d'après la définition ci-dessus, ce qui revient à dire que $\Hom_\mathcal{D}(X,Y)=\Hom_\mathcal{C}(X,Y)$ pour tout couple $(X,Y)$ d'objets dans $\mathcal{D}.$ 

\medskip

\textbf{Sous-foncteur.} Soit $F$ un foncteur d'une catégorie $\mathcal{C}$ vers les ensembles. On dit qu'un foncteur $G: \mathcal{C} \longrightarrow \mathfrak{Ens}$ est un sous-foncteur de $F$, et on le note $G \subset F$ par analogie avec les ensembles, si et seulement si 
\begin{enumerate}
\item pour tout objet $X$ dans $\mathcal{C}$, $F(X)=G(X)$ et
\item pour tout morphisme $f: X \longrightarrow Y$, $G(f)=F(f)_{|G(X)}$.  
\end{enumerate}

% *************************
% Appel de la bibliographie
% *************************
%\bibliographystyle{alpha}
%\bibliography{$HOME/repertoiredufichierdebiblio}

\bibliographystyle{alpha}

\end{document}